\documentclass[12pt]{article}

\usepackage{algorithm}
\usepackage{algorithmic}
\usepackage{amsfonts,dsfont}
\usepackage{hyperref}
\usepackage{cite}
\usepackage{color}
\usepackage{amsmath, amssymb}
\usepackage{amsthm}
\usepackage{bm}
\usepackage{color}
\usepackage{comment}
\usepackage{geometry}
\usepackage{tabularx}
\usepackage{graphicx}
\graphicspath{{figs/}}
\usepackage{subfigure}
\usepackage{placeins}

\allowdisplaybreaks

\hypersetup{colorlinks, linkcolor=blue, citecolor=blue, urlcolor=blue}

\theoremstyle{definition}\newtheorem{definition}{Definition}[section]
\theoremstyle{definition}\newtheorem{proposition}[definition]{Proposition}
\theoremstyle{definition}\newtheorem{remark}[definition]{Remark}
\theoremstyle{definition}\newtheorem{theorem}[definition]{Theorem}
\theoremstyle{definition}
\theoremstyle{definition}
\theoremstyle{definition}\newtheorem{example}[definition]{Example}
\theoremstyle{definition}
\theoremstyle{definition}

\newcommand{\bE}{\textbf{E}}
\newcommand{\bS}{\textbf{S}}

\newcommand{\J}{\mathrm{J}}
\newcommand{\td}{\text{d}}

\newcommand{\revise}[1]{{\color{blue}#1}}

\newcommand{\mK}{\mathcal{K}}
\newcommand{\n}{^{(n)}}
\newcommand{\oL}{\overline{L}}

\newcommand{\pp}[2]{\frac{\partial #1}{\partial #2}}
\newcommand{\pr}{\text{pr}\,}

\newcommand{\sm}{^{[m]}}
\newcommand{\sn}{^{[n]}}

\newcommand{\vv}{\mathbf{v}}

\newcommand{\steps}{\operatorname{steps}}
\newcommand{\EOC}[1]{\operatorname{EOC}\left(#1\right)}
\newcommand{\enorm}[1]{\ensuremath{\left\| #1 \right\|_{l_{\infty}}}}
\newcommand{\err}[1]{\ensuremath{e_{#1}}}
\newcommand{\abs}[1]{\ensuremath{\left|#1\right|}}
\renewcommand{\max}[2]{\ensuremath{\underset{#1}{\operatorname{max}}\left(#2\right)}}

\begin{document}
\begin{center}
{\Large{\bf Invariant Variational Schemes\\ for Ordinary Differential Equations}}\\
\end{center}
\vspace{1cm}
\begin{minipage}[t]{0.5\linewidth}
Alex Bihlo\\
Department of Mathematics and Statistics\\ 
Memorial University of Newfoundland\\
St.\ John's, NL, A1C 5S7, Canada\\
 {\tt abihlo@mun.ca}
\end{minipage}
\begin{minipage}[t]{0.5\linewidth}
James Jackaman\\
Department of Mathematical Sciences\\
Norwegian University of Science and Technology (NTNU) \\
7491 Trondheim, Norway \\
{\tt james.jackaman@ntnu.no}
\end{minipage}\\[0.5cm]
\begin{minipage}{0.75\linewidth}
Francis Valiquette \\
Department of Mathematics\\ 
Monmouth University\\ 
West Long Branch, NJ, 07764, USA\\
{\tt fvalique@monmouth.edu}
\end{minipage}\\[0.5cm]


\noindent {\bf MSC2020 Mathematics subject classification:} 58D19, 65L12, 65L50

\begin{abstract}
We propose a novel algorithmic method for constructing invariant variational schemes of systems of ordinary differential equations that are the Euler--Lagrange equations of a variational principle. The method is based on the invariantization of standard, non-invariant discrete Lagrangian functionals using equivariant moving frames. The invariant variational schemes are given by the Euler--Lagrange equations of the corresponding invariantized discrete Lagrangian functionals. We showcase this general method by  constructing invariant variational schemes of ordinary differential equations that preserve variational and divergence symmetries of the associated continuous Lagrangians.  Noether's theorem automatically implies that the resulting schemes are exactly conservative. Numerical simulations are carried out and show that these invariant variational schemes outperform standard numerical discretizations.
\end{abstract}

\section{Introduction} \label{sec:intro}

The aim of geometric numerical integration is to construct numerical schemes that preserve certain geometric features of differential equations. In doing so, geometric integrators typically provide better global and long term numerical results than comparable non-geometric methods.  Typical examples include, amongst others, symplectic integrators, \cite{BC-2016, HLW-2006, LR-2004,SC-1994}, Lie--Poison structure preserving schemes, \cite{ZM-1988}, energy-preserving methods, \cite{QM-2008}, exactly conservative schemes, \cite{WBN-2016,WBN-2017}, symmetry-preserving methods, \cite{B-2013,BJV-2020,BN-2013,BN-2014,BV-2019}, and variational integrators, \cite{MW-2001}.

In this paper we use the method of moving frames, \cite{MMW-2013,MM-2018,O-2001}, to construct numerical schemes for ordinary differential equations that preserve variational symmetries of Euler--Lagrange equations. An application of Noether's theorem implies that the resulting schemes are necessarily conservative and preserve the associated ``constants of motion."  Such schemes are constructed as follows.  Given the Euler--Lagrange equations $\bE(L)=0$, with variational symmetry group $G$, consider its Lagrangian functional $\mathcal{L}[u]=\int L\, dx$.  We note that this Lagrangian functional is not unique. It can always be scaled by a constant and one can always add a divergence term.  Nevertheless, since $G$ is a variational symmetry group, $\mathcal{L}$ can be chosen so that it is invariant under the action of $G$.  Next, introduce a finite difference approximation $\mathcal{L}^\td = \sum_k\, L_k$ of $\mathcal{L}$.  In general, $\mathcal{L}^\td$ will not be invariant under the product action of $G$.   To obtain a symmetry-preserving Lagrangian, we follow the general procedure in \cite{BV-2017,KO-2004} and compute the invariantized Lagrangian $\iota(\mathcal{L}^\td)$ using the method of equivariant moving frames.  The discrete Euler--Lagrange equations $\bE^\td(\iota(L_k)) = 0$ are then used to approximate the original equations $\bE(L)=0$.  Since $\bE^\td(\iota(L_k)) = 0$ is invariant under the action of $G$, Noether's theorem implies that the scheme is conservative and preserves the constants of motion.  The above procedure can be modified to deal with Lagrangian functionals that admit divergence (Bessel-Hagen) symmetries.  In this case, it suffices to modify the Lagrangian in such a way that divergence symmetries become variational symmetries.

The proposed methodology is related to various other approaches used in geometric numerical integration. As the schemes developed in~\cite{MW-2001}, the proposed methodology is \textit{variational}, in that we discretize the Lagrangian rather than the associated Euler--Lagrange equations. Furthermore, as the schemes proposed in~\cite{B-2013,BJV-2020,BN-2013,BN-2014,BV-2019}, our schemes are \textit{invariant} as well, due to the well-known fact that symmetries of a Lagrangian are also symmetries of the corresponding Euler--Lagrange equations, \cite{O-1993}. Lastly, similar to the exactly conservative schemes derived in \cite{WBN-2016,WBN-2017}, our schemes will also be exactly conservative, thanks to Noether's theorem.  Therefore, our methodology combines several geometric features into one numerical integrator.

The organization of the paper is as follows. In Section \ref{sec:invariant-L} we begin by recalling standard results concerning variational problems and their symmetry groups.   In particular, in Section \ref{ref:div-sym} we explain how a divergence symmetry group can be made into a variational symmetry group by appropriately modifying the Lagrangian.  In Section \ref{sec:discrete-L} we review the theory of discrete variational problems, their symmetries, and Noether's Theorem.  To construct symmetry-preserving discrete Lagrangians, and therefore invariant Euler--Lagrange equations, we use the method of equivariant moving frames, which is summarized in Section \ref{sec:mf}.  In Section \ref{sec-sym-preserving-schemes} we outline the procedure for constructing conservative schemes of Euler--Lagrange equations that preserve their variational/divergence symmetries.  Finally, in Section \ref{sec:num-sim} numerical simulations are carried out that verify numerically the exact conservative nature of the proposed invariant variational schemes.  Also, when compared to ``standard schemes," invariant variational schemes provide better long term numerical results.

\section{Invariant Lagrangians}\label{sec:invariant-L}

In this section we recall standard results concerning invariant variational problems and, more generally, divergent invariant variational problems.  For a detailed exposition, we refer the reader to \cite{O-1993}.  We begin by introducing some notation and terminology.

In this paper we consider ordinary differential equations and let $x \in \mathbb{R}$ denote the independent variable. If $u=(u^1,\ldots,u^q) \in \mathbb{R}^q$ represent the dependent variables, then the \emph{$n$\textsuperscript{th} order jet space of curves} in $\mathbb{R}^q$, denoted $\J\n=\J\n(\mathbb{R}^q,1)$, is locally parametrized by $(x,u\n)$, where $u\n=(u,u_x,\ldots,u_{x^n})$ collects the derivatives $u_{x^k}$ of order $0\leq k \leq n$.

Let $G$ be an $r$-dimensional Lie group acting on $(x,u)\in \mathbb{R}^{q+1}$:
\[
X = g\cdot x,\qquad U=g\cdot u.
\]
The induced action on the horizontal one-form $\mathrm{d}x$ yields the \emph{lifted} horizontal form
\[
\omega = \mathrm{D}_x(X)\, \mathrm{d}x,
\]
where
\[
\mathrm{D}_x = \pp{}{x} + \sum_{\alpha=1}^q \sum_{k\geq 0} u^\alpha_{x^{k+1}}\pp{}{u^\alpha_{x^k}}
\]
denotes the total derivative operator.  

\begin{remark}
More precisely, the lift of $\mathrm{d}x$ should be
\[
\omega = \mathrm{D}_x(X)\, \mathrm{d}x + \sum_{\alpha=1}^q X_{u^\alpha} \theta^\alpha,
\]
where $\theta^\alpha = \mathrm{d}u^\alpha - u^\alpha _x\, \mathrm{d}x$ are the order zero basic contact one-forms, \cite{KO-2003}.  However, since our computations are performed modulo contact forms, these are omitted.
\end{remark}

Dual to $\omega$, we have the \emph{lifted derivative operator}
\[
\mathrm{D}_X = \frac{1}{\mathrm{D}_x(X)}\, \mathrm{D}_x.
\]
The \emph{prolonged action} of $G$ to the $n$\textsuperscript{th} order jet space $\J\n$ is given by
\[
U^\alpha_{X^k} = \mathrm{D}_X^k(U^\alpha),\qquad \alpha = 1,\ldots, q,\qquad 0\leq k\leq n.
\]
At the infinitesimal level, let
\[
\vv_\nu = \xi_\nu(x,u)\pp{}{x} + \sum_{\alpha=1}^q \phi^\alpha_\nu \pp{}{u^\alpha},\qquad \nu=1,\ldots,r,
\]
denote a basis of infinitesimal generators.  The prolongation formula for the infinitesimal generators is
\[
\pr\vv_\nu = \xi_\nu(x,u)\pp{}{x} + \sum_{k\geq 0} \sum_{\alpha=1}^q \phi^{\alpha,k}_\nu \pp{}{u^\alpha_{x^k}},\qquad \nu=1,\ldots,r,
\]
with the component $\phi^{\alpha,k}_\nu$ given by the formula 
\[
\phi^{\alpha,k}_\nu = \mathrm{D}_x^k(Q^\alpha_\nu) + \xi_\nu u^{\alpha}_{x^{k+1}},\qquad\text{where} \qquad Q^\alpha_\nu(x,u^{(1)}) = \phi^\alpha_\nu -  \xi_\nu u_x^\alpha
\]
are the components of the characteristic $Q_\nu=(Q^1_\nu,\ldots,Q^q_\nu)$.

\begin{example}
Consider the action of the special Euclidean group $\text{SE}(2)$ on planar curves $\{(x,u(x))\}$ given by
\begin{equation}\label{SE(2)}
X = x \cos\varphi - u\sin \varphi + a,\qquad
U = x\sin \varphi + u \cos \varphi + b,
\end{equation}
where $a, b, \varphi \in \mathbb{R}$.  Then the horizontal lifted one-form is
\[
\omega =  \mathrm{D}_x(X)\, \mathrm{d}x = (\cos \varphi - u_x \sin \varphi) \, \mathrm{d}x,
\]
and the lifted derivative operator is
\[
\mathrm{D}_X = \frac{1}{\cos \varphi - u_x \sin \varphi} \mathrm{D}_x.
\]
Therefore, the prolonged action is, up to order two, 
\[
U_X = \frac{\sin \varphi + u_x \cos\varphi}{\cos\phi - u_x \sin \varphi},\qquad
U_{XX} = \frac{u_{xx}}{(\cos\varphi - u_x \sin\varphi)^3}.
\]
A basis of infinitesimal generators is given by the vector fields
\begin{equation}\label{eq:elastica-v}
\vv_1 = \pp{}{x},\qquad \vv_2=\pp{}{u},\qquad \vv_3 = -u\pp{}{x} + x\pp{}{u}.
\end{equation}
Up to order two, their prolongation is
\[
\pr^{(2)}\vv_1= \pp{}{x},\qquad \pr^{(2)}\vv_2 = \pp{}{u},\qquad
\pr^{(2)}\vv_3 = -u \pp{}{x} + x \pp{}{u} + (1+u_x^2)\pp{}{u_x} + 3u_x u_{xx} \pp{}{u_{xx}}.
\]
\end{example}
 
 \subsection{Variational Symmetry}
 
We now recall the notion of a variational symmetry group for a Lagrangian and the celebrated Noether's theorem.
 
\begin{definition}
A connected Lie group of transformations $G$ acting on $\mathbb{R}^{q+1}$ is called a \emph{variational symmetry group} of the functional $\mathcal{L}[u] = \int L(x,u^{(n)})\,\mathrm{d}x$ if and only if
\[
\int g\cdot (L(x,u\n)\, \mathrm{d}x) = \int L(X,U\n)\, \omega = \int L(x,u\n)\, \mathrm{d}x\qquad\text{for all}\qquad g\in G,
\]
where the prolonged action is defined. At the infinitesimal level, if $\vv_1,\ldots,\vv_r$ is a basis of infinitesimal generators, then $G$ is a variational symmetry group of the functional if and only if
\[
\text{pr}^{(n)} \vv_\nu(L) + L\, \mathrm{D}_x(\xi_\nu) = 0,\qquad \nu=1,\ldots,r.
\]
\end{definition}
 
 \begin{example}
A classical example of $\text{SE}(2)$-invariant Lagrangian is given by the Euler elastica
\begin{equation}\label{eq: elastica functional}
\mathcal{L} = \int \frac{1}{2} \kappa^2\, \mathrm{d}s,
\end{equation}
where $\kappa = \dfrac{u_{xx}}{(1+u_x^2)^{3/2}}$ is the curvature of a planar curve and $\omega = \mathrm{d}s = \sqrt{1+u_x^2}\, \mathrm{d}x$ is the arc-length element.  In local coordinates, the functional \eqref{eq: elastica functional} is
\begin{equation}\label{eq:elastica-functional}
\mathcal{L}[u]= \int \frac{u_{xx}^2}{2(1+u_x^2)^{5/2}}\, \mathrm{d}x.
\end{equation}
The elastica problem has a long history dating back to Euler, \cite{E-1744}.  For a more modern account we refer to \cite{L-1927}.
\end{example}

\begin{definition}
For $1\leq \alpha \leq q$, the $\alpha$\textsuperscript{th} \emph{Euler operator} is the differential operator
\begin{equation}\label{eq:Euler-operator}
\bE_\alpha = \sum_{k=0}^\infty (-{\rm D}_x)^k \pp{}{u^\alpha_{x^k}}.
\end{equation}
\end{definition}

\begin{proposition}
If $u=u(x)$ is a smooth extremal of the variational problem $\mathcal{L}[u] = \int L(x,u\n) \mathrm{d}x$, then it must be a solution of the Euler--Lagrange equations
\[
\bE(L) = (\bE_1(L),\ldots,\bE_q(L))=0.
\]
\end{proposition}

\begin{example}
The Euler--Lagrange equation of the Euler elastica functional \eqref{eq:elastica-functional} is
\begin{equation}\label{eq: xu elastica}
2 u_{xxxx}(1+u_x^2)^2+5u_{xx}^3(6 u_x^2-1)-20 u_x u_{xx} u_{xxx}(1+u_x^2)=0.
\end{equation}
In terms of the curvature and its arc length derivative, the differential equation \eqref{eq: xu elastica} simplifies to
\begin{equation}\label{eq: free elastica}
\kappa_{ss} + \frac{\kappa^3}{2} = 0.
\end{equation}
\end{example}

\begin{theorem}\label{thm:symEL}
If $G$ is a variational symmetry group of $\mathcal{L}[u]=\int L(x,u\n)\, \mathrm{d}x$, then $G$ is a symmetry group of the Euler--Lagrange equations $\bE(L)=0$.
\end{theorem}

\begin{remark}
Since \eqref{eq: xu elastica} is expressible in terms of the curvature and its arc length derivatives, the differential equation is immediately invariant under the prolonged action of the special Euclidean group $\text{SE}(2)$.

We note that the converse of Theorem \ref{thm:symEL} is incorrect, \cite{O-1993}.  In general, $\bE(L)=0$ can admit a larger symmetry group than that of the functional $\mathcal{L}[u] = \int L(x,u\n)\, \mathrm{d}x$.
\end{remark}
 
\begin{definition}
A \emph{conserved quantity} (or \emph{constant of motion} or \emph{first integral}) for the system of ordinary differential equations $\Delta(x,u\n)=0$ is a  function $C(x,u^{(m)})$ such that
\[
\mathrm{D}_x(C) = 0
\]
on the solution space of $\Delta(x,u\n)=0$.  In other words, $C(x,u^{(m)})$ is constant on solutions of $\Delta(x,u\n)=0$.
\end{definition}

We now state one of the simplest versions of Noether's Theorem, \cite{O-1993}.

\begin{theorem}\label{thm:Noether}
Let $G$ be a one-parameter group of variational symmetries for the functional $\mathcal{L}[u] = \int L(x,u\n)\, \mathrm{d}x$ with infinitesimal generator
\begin{equation}\label{eq:v}
\vv = \xi(x,u)\pp{}{x} + \sum_{\alpha=1}^q \phi_\alpha(x,u)\pp{}{u^\alpha}
\end{equation}
and characteristic components $Q^\alpha(x,u^{(1)}) = \phi_\alpha -  \xi u_x^\alpha$.  Then, there exists a constant of motion $C=-(A+L\xi)$ where $A$ is a certain function depending on $Q=(Q^1,\ldots,Q^q)$, $L$, and their derivatives.
\end{theorem}

\begin{example}
For a first order variational problem $\mathcal{L}[u] = \int L(x,u^{(1)})\, \mathrm{d}x$, with infinitesimal variational symmetry generator \eqref{eq:v},
\begin{equation}\label{eq: Noether order 1}
C = -\bigg(\sum_{\alpha=1}^q Q_\alpha \pp{L}{u^\alpha_x} + \xi L\bigg)
\end{equation}
is a conserved quantity of the Euler--Lagrange equations $\bE(L)=0$.
\end{example}

\begin{example}
For a one-dimensional variational problem of order two, $\mathcal{L}[u] = \int L(x,u^{(2)})\, \mathrm{d}x$, with infinitesimal variational symmetry generator \eqref{eq:v}, 
\begin{equation}\label{eq:order2-C}
C = -\bigg(Q\bigg(\pp{L}{u_x} - D_x\bigg(\pp{L}{u_{xx}}\bigg)\bigg) + D_x(Q)\pp{L}{u_{xx}}+\xi L\bigg),
\end{equation}
is a conserved quantity of the Euler--Lagrange equations $\bE(L)=0$.
\end{example}

\begin{example}
For the Euler elastica functional \eqref{eq:elastica-functional}, the conserved quantities that come from  \eqref{eq:order2-C} are
\begin{gather*}
C_1 = \frac{2u_x u_{xxx}}{(1+u_x^2)^{5/2}}-\frac{(1+6u_x^2)u_{xx}^2}{(1+u_x^2)^{7/2}},\qquad
C_2 = \frac{5 u_x u_{xx}^2}{(1+u_x^2)^{7/2}} - \frac{2u_{xxx}}{(1+u_x^2)^{5/2}},\\
C_3 =  (x+uu_x)\bigg[\frac{2u_{xxx}}{(1+u_x^2)^{5/2}}-\frac{5u_xu_{xx}^2}{(1+u_x^2)^{7/2}}\bigg]-\frac{uu_{xx}^2}{(1+u_x^2)^{5/2}}-\frac{2u_{xx}}{(1+u_x^2)^{3/2}}.
\end{gather*}
\end{example}

\subsection{Divergence Symmetry}\label{ref:div-sym}

The notion of variational symmetry was extended by Bessel-Hagen, \cite{B-1921}, to allow  divergence symmetries of a variational functional, \cite{O-1993}.

\begin{definition}\label{def:div-sym-group}
A connected Lie group of transformations $G$ acting on $\mathbb{R}^{q+1}$ is called a \emph{divergence symmetry group} of the functional $\mathcal{L}[u] = \int L(x,u\n)\, \mathrm{d}x$ if and only if 
\begin{equation}\label{eq:divergence-invariance}
\int g\cdot (L(x,u\n)\, \mathrm{d}x) = \int [L(x,u\n)+\mathrm{D}_x(P_g(x,u\n))]\, \mathrm{d}x,
\end{equation}
for some differential function $P_g(x,u\n)$ depending on the group parameter $g\in G$.  At the infinitesimal level, if $\vv_1,\ldots, \vv_r$ is a basis of infinitesimal generators of $G$, then $G$ is a divergence symmetry group if and only if
\[
\pr\n\vv_\nu(L) + L\, \mathrm{D}_x(\xi_\nu) = \mathrm{D}_x(B_\nu),\qquad \nu=1,\ldots,r,
\]
where $B_\nu(x,u\n)$ are certain differential functions.
\end{definition}

Since the kernel of the Euler--Lagrange operators \eqref{eq:Euler-operator} are total derivatives of differential functions, i.e.\ $\mathrm{D}_x(B(x,u\n))=0$, it follows that divergence symmetries produce symmetries of the corresponding Euler--Lagrange equations $\bE(L)=0$. Noether's Theorem \ref{thm:Noether} still holds for divergence symmetries.  Constants of motions are now given by $C=B-A-L\xi$.  

We now show that any divergence symmetry group $G$ of a variational problem $\mathcal{L}[u]=\int L(x,u\n)\, \mathrm{d}x$ can be made into the variational symmetry group of a modified Lagrangian with identical Euler--Lagrange equations.  This observation will play an important role in Section~\ref{sec-sym-preserving-schemes}.  

\begin{theorem}
Let $\mathcal{L}[u] = \int L(x,u\n)\, \mathrm{d}x$ be a functional with divergence symmetry group $G$ satisfying \eqref{eq:divergence-invariance}.   Then $G$ is a variational symmetry group of the modified functional
\begin{equation}\label{eq:modified-lagrangian}
\overline{\mathcal{L}}[u] = \int \oL\, \mathrm{d}x = \int (L+\zeta_x)\, \mathrm{d}x,
\end{equation}
where $G$ acts on the new variable $\zeta$ according to
\begin{equation}\label{eq:zeta action}
g\cdot \zeta = \zeta - P_g.
\end{equation}
\end{theorem}

\begin{proof}
We first show that \eqref{eq:zeta action} induces a well-defined left group action on $\mathrm{D}_x(\zeta)=\zeta_x$.  To this end, let $h,g\in G$. We first note that
\begin{align*}
\mathrm{D}_x(P_{hg})\, \mathrm{d}x &= (hg)\cdot (L\, \mathrm{d}x) - L\,\mathrm{d}x \\
&= h\cdot (L+\mathrm{D}_x(P_g))\, \mathrm{d}x-L\, \mathrm{d}x \\
&= (L+\mathrm{D}_x(P_h))\, \mathrm{d}x +h\cdot [\mathrm{D}_x(P_g)\, \mathrm{d}x] - L\, \mathrm{d}x  \\
&= \mathrm{D}_x(P_h)\,\mathrm{d}x + \mathrm{D}_{X} (h\cdot P_g)\, \omega \\
&= \mathrm{D}_x(P_h)\,\mathrm{d}x + \mathrm{D}_x( h\cdot P_g)\, \mathrm{d}x \\
&= \mathrm{D}_x(P_h+h \cdot P_g)\, \mathrm{d}x.
\end{align*}
Thus
\begin{align*}
(hg)\cdot [\mathrm{D}_x(\zeta)\, \mathrm{d}x] &= \mathrm{D}_X [(hg)\cdot \zeta]\, \omega \\
&= \mathrm{D}_x[(hg)\cdot \zeta]\, \mathrm{d}x \\
&= \mathrm{D}_x(\zeta - P_{hg})\, \mathrm{d}x \\
&= \mathrm{D}_x(\zeta - P_h - h\cdot P_g)\, \mathrm{d}x \\
&= \mathrm{D}_x(h\cdot (\zeta-P_g))\, \mathrm{d}x \\
&= h\cdot(\mathrm{D}_x(g\cdot \zeta)\, \mathrm{d}x) \\
&= h\cdot (g\cdot (\mathrm{D}_x(\zeta)\, \mathrm{d}x)),
\end{align*}
which shows that we have a well-defined left group action on $\mathrm{D}_x(\zeta)=\zeta_x$.  

It is now straightforward to show that $G$ is a variational symmetry group of the modified Lagrangian functional \eqref{eq:modified-lagrangian}.  For $g\in G$,
\begin{align*}
\int g\cdot (\oL\, \mathrm{d}x) &= \int g\cdot (L\, \mathrm{d}x) + \int g\cdot (\zeta_x\, \mathrm{d}x) \\
&= \int (L+\mathrm{D}_x(P_g))\, \mathrm{d}x+ \int \mathrm{D}_X(g\cdot \zeta)\, \omega \\
&= \int (L+\mathrm{D}_x(P_g))\, \mathrm{d}x+ \int (\mathrm{D}_x(X))^{-1} (\zeta_x-\mathrm{D}_x(P_g)) \, (\mathrm{D}_xX)\, \mathrm{d}x \\
&= \int (L+\mathrm{D}_x(P_g))\, \mathrm{d}x+ \int (\zeta_x-\mathrm{D}_x(P_g))\, \mathrm{d}x \\
&= \int (L+\zeta_x)\, \mathrm{d}x \\
&= \int \oL\, \mathrm{d}x.
\end{align*}
\end{proof}

\begin{remark}
By construction, we note that $\mathcal{L}=\int L\, \mathrm{d}x$ and $\overline{\mathcal{L}} = \int \overline{L}\,\mathrm{d}x$ have the same conserved quantities.
\end{remark}

\begin{example}\label{ex:divergence invariant lagrangian}
A simple example of Lagrangian admitting a divergence symmetry group is given by
\[
\mathcal{L} = \int L\, \mathrm{d}x = \int \bigg(u_x^2 - \frac{1}{u^2}\bigg)\, \mathrm{d}x,
\]
with Euler--Lagrange equation
\begin{equation}\label{eq:divergent EL}
u_{xx} = \frac{1}{u^3}.
\end{equation}
The corresponding divergence symmetry group action is
\[
X = \frac{\alpha x + \beta}{\delta x + \gamma},\qquad U= \frac{u}{\delta x + \gamma},\qquad \text{where}\qquad \alpha\gamma - \beta \delta = 1.
\]
The associated infinitesimal generators are
\begin{equation}\label{eq:v-ex2}
\vv_1 = \pp{}{x},\qquad \vv_2 = 2x\pp{}{x} + u\pp{}{u},\qquad \vv_3 = x^2\pp{}{x}+xu\pp{}{u}.
\end{equation}
We note that the first two vector field generate variational symmetries since
\[
\pr\vv_1(L) + L \mathrm{D}_x(\xi_1) = 0,\qquad \pr\vv_2(L) + L \mathrm{D}_x(\xi_2) = 0.
\]
On the other hand,
\[
\pr\vv_3(L) + L \mathrm{D}_x(\xi_3) = 2uu_x = \mathrm{D}_x(u^2),
\]
which induces a divergence symmetry.  Using \eqref{eq: Noether order 1}, the corresponding conserved quantities are
\[
C_1 = u_x^2+\frac{1}{u^2}, \qquad
C_2 = 2\frac{x}{u^2} - 2(u-x u_x)u_x,\qquad
C_3 = \frac{x^2}{u^2} + (u-xu_x)^2.
\]
These constants of motion are not independent and satisfy the equation
\begin{equation}\label{eq:C-relation}
\frac{C_2^2}{4} - C_1C_3+1=0.
\end{equation}

Since
\begin{align*}
\int g\cdot (L\, \mathrm{d}x) &= \int \bigg[((\delta x+\gamma)u_x - \delta u)^2 - \frac{(\delta x+ \gamma)^2}{u^2}\bigg]\, \frac{\mathrm{d}x}{(\delta x+\gamma )^2} \\
&= \int \bigg[ u_x - \frac{2\delta uu_x}{\delta x+\gamma} + \frac{\delta^2 u^2}{(\delta x+\gamma)^2} - \frac{1}{u^2}\bigg]\, \mathrm{d}x \\
&= \int \bigg[ L + \mathrm{D}_x\bigg(-\frac{\delta u^2}{\delta x+\gamma}\bigg)\bigg]\, \mathrm{d}x,
\end{align*}
an invariant Lagrangian can be defined by introducing a new variable $\zeta$ such that
\[
g\cdot \zeta = \zeta + \frac{\delta u^2}{\delta x + \beta}.
\]
The induced prolonged action is 
\[
g\cdot \zeta_x= \mathrm{D}_X \bigg( \zeta + \frac{\delta u^2}{\delta x + \beta}\bigg) = (\delta x+\gamma)^2 \mathrm{D}_x\bigg( \zeta + \frac{\delta u^2}{\delta x + \beta}\bigg) = (\delta x+\gamma)^2 \bigg[\zeta_x +\mathrm{D}_x\bigg(\frac{\delta u^2}{\delta x + \beta}\bigg)\bigg],
\]
and the modified functional 
\begin{equation}\label{eq:modified lagrangian}
\int \oL\, \mathrm{d}x= \int (L+\zeta_x)\, \mathrm{d}x = \int\bigg(u_x^2-\frac{1}{u^2}+\zeta_x\bigg)\, \mathrm{d}x
\end{equation}
is, by construction, invariant.
\end{example}

\section{Discrete Lagrangians}\label{sec:discrete-L}

We now adapt the results of the previous section to the discrete setting.  Let $z=(z^0,\ldots,z^{q}) = (x,u)$ be coordinates on $\mathbb{R}^{q+1}$.  In this section we are concerned with discrete $\mathbb{R}^{q+1}$-valued functions
\begin{equation}\label{discrete function}
f \colon \mathbb{Z} \to \mathbb{R}^{q+1},\qquad  k \mapsto f(k)=(f^0(k),\ldots,f^q(k)).
\end{equation}
As it is customarily done, we use the index notation
\[
f_k = f(k)
\]
to denote the value of $f$ at $k \in \mathbb{Z}$.  Introducing the \emph{lattice variety}
\[
\pi\colon \mathbb{Z} \times \mathbb{R}^{q+1} \to \mathbb{Z},
\]
the discrete map \eqref{discrete function} defines a one-dimensional {\it discrete submanifold}
\[
\{(k,f_k)\,|\, k\in \mathbb{Z}\} \subset \mathbb{Z} \times \mathbb{R}^{q+1}.
\]

The lattice space $\mathbb{Z}$ does not admit a differentiable structure.  Only the fibers $\pi^{-1}(k) = \mathbb{R}^{q+1}$ are smooth manifolds.  In the following, we use $z_k=(z_k^0,\ldots,z_k^q)$ as coordinates on $\pi^{-1}(k) = \mathbb{R}^{q+1}$.  Natural operators on $\mathbb{Z}$ are the \emph{forward shift}
\begin{equation}\label{eq: forward shift}
\bS = \bS^+ \colon k \mapsto k + 1,
\end{equation}
and the \emph{backward shift}
\[
\bS^{-}\colon k \mapsto k - 1.
\] 
The action of the shift maps $\bS^\pm$ on the fiber coordinates $z_k$ is
\[
\bS^\pm[z_k] = z_{k \pm 1}.
\]
Using the forward shift \eqref{eq: forward shift} we define the \emph{forward difference operator}
\[
\Delta = \bS - \mathds{1},
\]
where $\mathds{1}\colon \mathbb{Z} \to \mathbb{Z}$ is the identity transformation.   

\begin{definition}
Let $n_1 \leq n_2$ be two integers.  The order $n=n_2-n_1$ \emph{discrete jet space} is the lattice variety
\[
\J\sn_{n_1,n_2}=\mathbb{Z} \times (\mathbb{R}^{q+1})^{\times (n+1)},
\]
with coordinates
\[
z_k\sn = (k,\,\ldots\,z_{k+\ell}\, \ldots\,)\in \mathbb{Z}\times (\mathbb{R}^{q+1})^{\times (n+1)},
\]
where $n_1\leq \ell \leq n_2$.  When $n_1=0$ and $n_2=n\in \mathbb{N}_0$, we obtain what we call the $n$\textsuperscript{th} order \emph{forward discrete jet space} $\J\sn = \J\sn_{0,n}$ and drop the subscript notation.
\end{definition}

\begin{example}
For example, coordinates for $\J^{[2]} = \J^{[2]}_{0,2}$ are given by $z^{[2]}_k = (k,z_k,z_{k+1},z_{k+2})$, while coordinates for $\J^{[4]}_{-2,2}$ are provided by $z^{[4]}_k = (k,z_{k-2},z_{k-1},z_k,z_{k+1},z_{k+2})$.
\end{example}

\begin{definition}
Let $L\colon \J\sn \to \mathbb{R}$ be a discrete function.  A \emph{discrete functional} is a formal sum
\[
\mathcal{L}^\td[z] = \sum_{k\in \mathbb{Z}} L(z_k\sn) = \sum_k L_k.
\]
In the following we use the short-hand notation $L_k$ to denote $L(z_k\sn)$ and omit the range of summation over the integer $k \in \mathbb{Z}$.
\end{definition}

\begin{definition}
Let $\mathcal{F}(\J\sn)$ denote the space of real-valued discrete functions $F\colon \J\sn \to \mathbb{R}$.  For $0\leq \alpha \leq q$, the $\alpha$\textsuperscript{th} \emph{discrete Euler operator} is the differential-difference operator $\bE^\td_\alpha\colon \mathcal{F}(\J\sn) \to \mathcal{F}(\J^{[2n]}_{-n,n})$ given by
\[
\bE_\alpha^\td = \sum_{0\leq \ell \leq n} \bS^{-\ell}\pp{}{z^\alpha_{k+\ell}}.
\]
\end{definition}

\begin{theorem}
If $z_k$ is an extremal of the discrete functional $\mathcal{L}^\td[z] = \sum_{k} L_k $, then it must be a solution of the discrete Euler--Lagrange equations
\[
E_\alpha^\td(L_k) = 0,\qquad \alpha=0,\ldots,q.
\]
\end{theorem}

Now let $G$ be a Lie group acting on $z_k$.  The prolonged action to $z_k\sn$ is given by the product action
\[
g\cdot z_k\sn = (k,\ldots\, g\cdot z_{k+\ell}\,\ldots).
\]
We note that the Lie group $G$ does not act in the discrete variable $k \in \mathbb{Z}$.  Thus, the action is well-defined on each fiber $\pi^{-1}_n(k) = \J\sn_{n_1,n_2}\big|_k$.

\begin{definition}
A Lie group of transformations $G$ is said to be a \emph{variational symmetry group} of the discrete functional $\displaystyle \mathcal{L}^\td[z] = \sum_{k} L(z_k\sn)$ if and only if
\[
g\cdot L(z_k\sn) = L(g\cdot z_k\sn) = L(z_k\sn).
\]
At the infinitesimal level, let
\[
\vv_\nu = \sum_{\alpha,k} Q_{\nu,k}^\alpha \pp{}{z^\alpha_k} = \sum_{\alpha,k} Q^\alpha(k,z_k) \pp{}{z^\alpha_k},\qquad \nu=1,\ldots,r,
\]
be a basis for the Lie algebra of infinitesimal generators of the group action.  Then $G$ is a variational symmetry group of $\displaystyle \mathcal{L}^\td[z]=\sum_k L_k$ if and only if
\[
\pr \vv(L_k) = \sum_{\alpha,\ell} Q_{\nu,k+\ell}^\alpha  \pp{L_k}{z^\alpha_{k+\ell}} = 0.
\]
\end{definition}

As in the continuous setting, Noether's Theorem still holds in the discrete setting, and each infinitesimal generator yields a conserved quantity.

\begin{definition}
Let $F_k=F\big(z\sn_k\big)=0$ be a system of finite difference equations. A \emph{conserved quantity} is a difference function $C_k=C(z_k\sm)$ such that
\[
\Delta (C_k)=0\qquad \text{on all solutions of}\qquad F_k=0.
\]
\end{definition}

\begin{theorem}\label{thm:discrete-Noether}
Let $\mathcal{L}^\td[z] = \sum_{k} L(z_k^{[1]})$ be a first order discrete Lagrangian with variational symmetry generator
\begin{equation}\label{eq:discrete-v}
\vv = \sum_{\alpha=0}^q Q^\alpha_k \pp{}{z^\alpha_k}.
\end{equation}
Then 
\[
C_k = \sum_{\alpha=0}^q Q_k^\alpha \pp{L_{k-1}}{z^\alpha_k}
\]
is a conserved quantity.
\end{theorem}

\begin{proof}
Since $\vv$ is a variational symmetry of $\mathcal{L}^\td[z]$,
\begin{align*}
0 &= \pr \vv(L_k) = \sum_{\alpha=0}^q Q^\alpha_k \pp{L_k}{z^\alpha_k} + Q^\alpha_{k+1} \pp{L_k}{z^\alpha_{k+1}} \\
&= \sum_{\alpha=0}^q \bigg[Q^\alpha_k \pp{L_k}{z^\alpha_k} + \bS \bigg(Q^\alpha_k \pp{L_{k-1}}{z^\alpha_k}\bigg)\bigg] \\
&= \sum_{\alpha=0}^q \bigg[Q^\alpha_k \pp{L_k}{z^\alpha_k} + \Delta\bigg(Q^\alpha_k \pp{L_{k-1}}{z^\alpha_k}\bigg) + Q_k^\alpha \pp{L_{k-1}}{z^\alpha_k}\bigg] \\
&= \sum_{\alpha=0}^q Q_k^\alpha \bE^\td_\alpha(L_k) + \Delta \bigg(\sum_{\alpha=0}^q Q^\alpha_k \pp{L_{k-1}}{z^\alpha_k}\bigg).
\end{align*}
Since $\bE^\td_\alpha(L_k)=0$, $\alpha = 0,\ldots,q$, the result follows.
\end{proof}

\begin{theorem}\label{thm:order2-discrete-noether}
Let $\mathcal{L}^\td[z] = \sum_{k} L(z_k^{[2]})$ be second order discrete Lagrangian  with variational symmetry generator \eqref{eq:discrete-v}. Then 
\[
C^\td = \sum_{\alpha=0}^q \bigg[Q_k^\alpha \pp{L_{k-1}}{z^\alpha_k} + Q_k^\alpha\pp{L_{k-2}}{z^\alpha_k} + Q^\alpha_{k+1} \pp{L_{k-1}}{z^\alpha_{k+1}}\bigg]
\]
is a conserved quantity.
\end{theorem}

\begin{remark}
As in Definition \ref{def:div-sym-group}, we can also introduce the notion of divergence symmetry in the discrete setting.  This more general notion of symmetry will not be used here since, as we have seen in the previous section, every divergence symmetry can be made into a variational symmetry by modifying the Lagrangian.
\end{remark}

As outlined in Section~\ref{sec:intro}, given a continuous Lagrangian functional $\mathcal{L}[u]$ with variational symmetry group $G$, our goal is to construct a discrete Lagrangian  $\mathcal{L}^\td[z]$ that will remain invariant under the action of $G$.  As the next example shows, in general, a standard discretization of $\mathcal{L}[u]$ will not preserve its symmetries.

\begin{example}
In an attempt to discretize the Euler elastica Lagrangian \eqref{eq:elastica-functional}, consider the discrete Lagrangian
\begin{equation}\label{eq:discrete-elastica-functional}
\mathcal{L}^\td = \sum_k L_k = \sum_k \frac{(u_{xx}^\td)^2}{2 (1+(u_x^\td)^2)^{5/2}} \cdot \sqrt{\Delta x_k \Delta x_{k+1}},
\end{equation}
where
\begin{equation}\label{eq:discrete-der}
u_x^\td = \frac{\Delta u_k}{\Delta x_k} = \frac{u_{k+1}-u_k}{x_{k+1}-x_k},\qquad u_{xx}^\td = \frac{1}{\sqrt{\Delta x_k \Delta x_{k+1}}}\bigg[\frac{\Delta u_{k+1}}{\Delta x_{k+1}} - \frac{\Delta u_k}{\Delta x_k}\bigg].
\end{equation}
One can verify that this discrete functional is invariant under translations, but not under rotations.
\end{example}

To construct a discrete Lagrangian functional $\mathcal{L}^\td[z]$ that will preserve the variational symmetries of a continuous Lagrangian $\mathcal{L}[u]$, we use the method of equivariant moving frames.

\section{Discrete Moving Frames}\label{sec:mf}

In this section we review the method of equivariant moving frames in the discrete setting.  We refer the reader to \cite{MMW-2013,MM-2018,O-2001} for a complete exposition of the method.

Let $G$ be an $r$-dimensional Lie group acting on $\mathbb{R}^{q+1}$, which is extended to $\J\sn$ via the product action. In the following, we assume that the action of $G$ on each fiber $\pi_n^{-1}(k)$ is (locally) free and regular, \cite{FO-1999}.  Recall that a Lie group $G$ acts freely on $\J\sn|_k = \pi_n^{-1}(k)$ if for all $z_k\sn \in \pi_n^{-1}(k)$ the isotropy subgroup $G_{z_k\sn} = \{g\,\in G\,|\, g\cdot z_k\sn = z_k\sn\}$ is trivial, i.e.\ $G_{z_k\sn} = \{e\}$.  The action is locally free if the isotropy subgroup $G_{z_k\sn}$ is discrete for all $z_k\sn \in \pi_n^{-1}(k)$.  This is equivalent to the fact that the orbits of the product group action have the same dimension as the group $G$.    By a result of Boutin, \cite{B-2002}, when the action of $G$ is (locally) effective on subsets of $\mathbb{R}^{q+1}$, local freeness on an open subset of $\pi_n^{-1}(k)$ can alway be achieved for a sufficiently large and finite $n$.  Finally, the action is regular if the orbits form a regular foliation.  When the action of $G$ on each fiber $\pi_n^{-1}(k)$ is (locally) free and regular, we say that $G$ acts (locally) freely and regularly on $\J\sn$.  

\begin{definition}
Let $G$ act (locally) freely and regularly on $\J\sn$.  A \emph{discrete (right) moving frame} is a $G$-equivariant map $\rho\colon \J\sn \to G$ satisfying
\begin{equation}\label{G-equivariance}
\rho(g\cdot z_k\sn) = \rho(z_k\sn)\, g^{-1},
\end{equation}
for all $g\in G$ where the product action is defined.
\end{definition}

To simplify the notation, we let 
\[
\rho_k = \rho(z_k\sn)
\]
denote the moving frame $\rho$ evaluated at the discrete jet $z_k\sn$.  In applications the construction of a (discrete) moving frame relies on the choice of a (discrete) cross-section $\mK \subset \J\sn$ to the group orbits.  

\begin{definition}
A subset $\mK \subset \J\sn$ is a \emph{discrete cross-section} to the group orbits if for each $k \in \mathbb{Z}$, the restriction $\mK|_k \subset \J\sn|_k = \pi^{-1}_n(k)$ is a submanifold of $\J\sn|_k$ transverse and of complementary dimension to the group orbits.
\end{definition}

In general, a cross-section $\mK \subset \J\sn$ is specified by a system of $r=\dim G$ difference equations
\[
\mK = \{E_\nu(z_n\sn) = 0\;|\; \nu = 1,\ldots,r\}.
\]
Once $\mK$ is fixed, the right moving frame at $z_k\sn$ is the unique group element $g=\rho_k=\rho(z_k\sn) \in G$ that sends $z_k\sn$ onto $\mK|_k$.  That is
\[
\rho_k \cdot z_k\sn \in \mK|_k.
\]
The coordinate expressions for the moving frame $\rho_k$ are obtained by solving the \emph{normalization equations}
\[
E_\nu(g\cdot z_k\sn) = 0,\qquad \nu=1,\ldots,r,
\]
for the group parameters $g=(g^1,\ldots,g^r)$.  

With a moving frame in hand, there is a systematic procedure, known as \emph{invariantization}, for constructing \emph{joint invariants} (also called \emph{discrete invariants} or \emph{difference invariants}).

\begin{definition}\label{def:invariantization}
The invariantization of the difference function $F(z_k\sn)$ is the joint invariant
\begin{equation}\label{invariantization}
\iota_k(F)(z_k\sn) = F(\rho_k\cdot z_k\sn).
\end{equation}
\end{definition}

The fact that the function in \eqref{invariantization} is invariant follows from the $G$-equivariant property \eqref{G-equivariance} that the right moving frame $\rho_k$ satisfies.  The operator $\iota_k$ is called the \emph{invariantization map} (with respect to $\rho_k$).  

Thus, given a discrete Lagrangian functional $\mathcal{L}^\td[z] = \sum_k L_k$ we can obtain a symmetry-preserving functional by invariantizing $\mathcal{L}^\td[z]$: 
\[
\iota(\mathcal{L}^\td[z]) = \sum_k \iota_k(L_k).
\]

\begin{example}
Consider the special Euclidean group action \eqref{SE(2)} acting on $z_k=(x_k,u_k)$:
\[
X_k = x_k \cos\varphi - u_k\sin \varphi + a,\qquad
U_k = x_k \sin \varphi + u_k \cos \varphi + b.
\]
A moving frame is obtained by selecting the cross-section
\[
\mathcal{K} = \{x_k=u_k = u_{k+1}=0\}.
\]
We observe that this cross-section is equivalent to 
\[
\mathcal{K} = \{x_k = u_k = u_x^\td = 0\},\qquad\text{where}\qquad u^\td_x = \frac{\Delta u_k}{\Delta x_k},
\]
the latter being a discrete approximation of the cross-section used in the continuous setting, \cite{KO-2003}. Solving the normalization equations $X_k = U_k = U_{k+1}=0$ for the group parameters $a$, $b$, $\varphi$, we obtain
\begin{equation}\label{eq: moving frame 1}
a = -\frac{x_k \Delta x_k + u_k \Delta u_k}{\ell_k},\qquad
b = \frac{x_k \Delta u_k - u_k \Delta x_k}{\ell_k},\qquad
\varphi = -\tan^{-1}\bigg(\frac{\Delta u_k}{\Delta x_k}\bigg),
\end{equation}
where
\[
\ell_k = \sqrt{\Delta x_k^2+\Delta u_k^2}.
\]
Using the invariantization map \eqref{invariantization} we have that
\[
\iota_k(\Delta x_k) = \ell_k\qquad \text{and}\qquad \iota_k(\Delta u_{k+1}) = \frac{D_k}{\ell_k},
\]
where
\[
D_k = \det \begin{bmatrix}
\Delta x_k & \Delta x_{k+1} \\
\Delta u_k & \Delta u_{k+1}
\end{bmatrix}.
\]
In the literature, and as in Definition \ref{def:invariantization}, it is customary to invariantize a discrete function $F(z\sn_k)$ with respect to $\iota_k$ solely.  In the following we expand this practice by using $\iota_k$ and $\iota_{k+1}$ simultaneously. For example, we invariantize $u_{xx}^\td$ given in \eqref{eq:discrete-der} as follows
\[
\iota(u_{xx}^\td) :=  \frac{1}{\sqrt{\iota_k(\Delta x_k) \iota_{k+1}(\Delta x_{k+1})}} \bigg[\frac{\iota_k(\Delta u_{k+1})}{\iota_{k+1}(\Delta x_{k+1})} - \frac{\iota_k(\Delta u_k)}{\iota_k(\Delta x_k)}\bigg] = \frac{1}{\sqrt{\ell_k\ell_{k+1}}}\cdot \frac{D_k}{\ell_k \ell_{k+1}} = \frac{D_k}{(\ell_k \ell_{k+1})^{3/2}}.
\]
We also invariantize $\Delta x_{k+1}$ using $\iota_{k+1}(\Delta x_{k+1}) = \ell_{k+1}$. Invariantizing the discrete Lagrangian functional \eqref{eq:discrete-elastica-functional}, we obtain
\begin{equation}\label{eq:discrete-elastica-invariant-L}
L_k^\iota=\iota(L_k)  = \frac{D_k^2}{2(\ell_k \ell_{k+1})^{5/2}}.
\end{equation}
Computing the corresponding discrete Euler--Lagrange equations yields
\begin{equation}\label{eq:discrete-elastica-EL}
\begin{aligned}
0 &=\bE_x^\td(L_k^\iota) = - \frac{D_{k-2} \Delta u_{k-2}}{(\ell_{k-2}\ell_{k-1})^{5/2}}+\frac{D_{k-1}(\Delta u_{k-1}+\Delta u_k)}{(\ell_{k-1}\ell_k)^{5/2}}-\frac{D_k\Delta u_{k+1}}{(\ell_k\ell_{k+1})^{5/2}}\\
&-\frac{5D_{k-2}^2\ell_{k-2}\Delta x_{k-1}}{4\ell_{k-1}(\ell_{k-2}\ell_{k-1})^{7/2}} -\frac{5 D_{k-1}^2 \ell_k \Delta x_{k-1}}{4\ell_{k-1}(\ell_{k-1}\ell_k)^{7/2}}+\frac{5D_{k-1}^2 \ell_{k-1}\Delta x_k}{4\ell_k(\ell_{k-1}\ell_k)^{7/2}}+\frac{5D_k^2 \ell_{k+1}\Delta x_k}{4\ell_k(\ell_k\ell_{k+1})^{7/2}},\\
0 &=\bE_u^\td(L_k^\iota) = \frac{D_{k-2}\Delta x_{k-2}}{(\ell_{k-2}\ell_{k-1})^{5/2}} -\frac{D_{k-1}(\Delta x_{k-1}+\Delta x_k)}{(\ell_{k-1}\ell_k)^{5/2}} +\frac{D_k\Delta x_{k+1}}{(\ell_k \ell_{k+1})^{5/2}} \\
&- \frac{5D_{k-2}^2\ell_{k-2}\Delta u_{k-1}}{4\ell_{k-1}(\ell_{k-2}\ell_{k-1})^{7/2}} 
-\frac{5D_{k-1}^2\ell_k\Delta u_{k-1}}{4\ell_{k-1}(\ell_{k-1}\ell_k)^{7/2}}+\frac{5D_{k-1}^2\ell_{k-1}\Delta u_k}{4\ell_k(\ell_{k-1}\ell_k)^{7/2}} 
+\frac{5D_k^2\ell_{k+1}\Delta u_k}{4\ell_k(\ell_k \ell_{k+1})^{7/2}}.
\end{aligned}
\end{equation}

Since the discrete Lagrangian \eqref{eq:discrete-elastica-invariant-L} is invariant under the special Euclidean group action, Noether's Theorem applies.  Using Theorem \ref{thm:order2-discrete-noether}, and recalling the infinitesimal generators \eqref{eq:elastica-v}, we obtain the conserved quantities
\begin{align*}
C_1^\td &= \pp{L_{k-1}^\iota}{x_k} + \pp{L_{k-2}^\iota}{x_k} + \pp{L_{k-1}^\iota}{x_{k+1}} \\
&= \frac{D_{k-1}\Delta u_k}{(\ell_{k-1}\ell_k)^{5/2}}-\frac{D_{k-2}\Delta u_{k-2}}{(\ell_{k-2}\ell_{k-1})^{5/2}} - \frac{5D_{k-1}^2\ell_k \Delta x_{k-1}}{4\ell_{k-1}(\ell_{k-1}\ell_k)^{7/2}}-\frac{5D_{k-2}^2\ell_{k-2}\Delta x_{k-1}}{4\ell_{k-1}(\ell_{k-2}\ell_{k-1})^{7/2}},\\
C_2^\td &= \pp{L_{k-1}^\iota}{u_k} + \pp{L_{k-2}^\iota}{u_k} + \pp{L_{k-1}^\iota}{u_{k+1}} \\
&= \frac{D_{k-2}\Delta x_{k-2}}{(\ell_{k-2}\ell_{k-1})^{5/2}} - \frac{D_{k-1}\Delta x_k}{(\ell_{k-1}\ell_k)^{5/2}}-\frac{5D_{k-1}^2\ell_k \Delta u_{k-1}}{4\ell_{k-1}(\ell_{k-1}\ell_k)^{7/2}}-\frac{5D_{k-2}^2 \ell_{k-2}\Delta u_{k-1}}{4\ell_{k-1}(\ell_{k-2}\ell_{k-1})^{7/2}},\\
C_3^\td &= - u_k \pp{L_{k-1}^\iota}{x_k} + x_k \pp{L_{k-1}^\iota}{u_k} - u_k \pp{L_{k-2}^\iota}{x_k} + x_k \pp{L_{k-2}^\iota}{u_k} - u_{k+1}\pp{L_{k-1}^\iota}{x_{k+1}} + x_{k+1} \pp{L_{k-1}^\iota}{u_{k+1}} \\
&= x_k C_2^\td - u_k C_1^\td + \frac{D_{k-1}}{(\ell_{k-1}\ell_k)^{5/2}}(\Delta x_{k-1}\Delta x_k + \Delta u_{k-1}\Delta u_k).
\end{align*}
\end{example}

\section{Invariant Variational Schemes}\label{sec-sym-preserving-schemes}

Given a continuous Lagrangian functional $\mathcal{L}[u]$, with Euler--Lagrange equations $\bE(L)=0$, we now describe a procedure for constructing a numerical scheme that will preserve its variational symmetries and thereby  be exactly conservative.

\begin{enumerate}
\item Let $\mathcal{L}[u]=\int L(x,u\n)\, \mathrm{d}x$ be a Lagrangian functional with variational symmetry group $G$, and let $\bE(L)=0$ be the corresponding Euler--Lagrange equations.

\item Introduce a discrete Lagrangian functional $\mathcal{L}^\td[z]=\sum_k L(z_k\sn)$, whose continuous limit is $\mathcal{L}[u]$.  In general $\mathcal{L}^\td[z]$ will not be invariant under the product action of $G$. \label{L discretization}

\item Assuming the product action is (locally) free and regular on $\J\sn$, 
 construct a discrete moving frame. As outlined in \cite{BV-2017}, and proved for curves in \cite{O-2001-1} and generalized in \cite{MM-2018}, for the discrete moving frame to have a well defined continuous limit, i.e. for the moving frame to converge to a differential moving frame and the discrete invariant Lagrangian and Euler--Lagrange equations to converge to their invariant differential counterparts, use a cross-section involving finite difference approximations of derivatives such as in \eqref{eq:discrete-der}. \label{moving frame}

\item Invariantize the discrete Lagrangian $\mathcal{L}^\td[z] = \sum_k L_k$ introduced in step \ref{L discretization} using the moving frame constructed in step \ref{moving frame}.

\item Compute the Euler--Lagrange equations $\bE^\td(\iota(L_k))=0$ of the invariantized Lagrangian $\iota(L_k)$.  These provide a numerical scheme approximating $\bE(L)=0$ that preserve the variational symmetry group $G$.  By Noether's Theorem, the numerical scheme $\bE^\td(\iota(L_k))=0$ also conserves the associated conserved quantities.
\end{enumerate}

If the Lagrangian functional $\mathcal{L}[u] = \int L(x,u\n)\, \mathrm{d}x$ admits a divergence symmetry group, the above steps still apply provided $\mathcal{L}[u]$ is replaced by the  modified functional $\overline{\mathcal{L}}[u] = \int \overline{L}\, \mathrm{d}x = \int (L+\zeta_x)\, \mathrm{d}x$ as described in Section \ref{ref:div-sym}.  

\begin{example}
To show how the above procedure works for a Lagrangian admitting a divergence symmetry group, let us continue Example \ref{ex:divergence invariant lagrangian}.  Starting from the modified Lagrangian functional \eqref{eq:modified lagrangian}, a possible discretization of $\overline{\mathcal{L}}[u]$ is
\begin{equation}\label{eq:discrete-modified-L}
\overline{\mathcal{L}}^\td = \sum_{k} \bigg[\bigg(\frac{\Delta u_k}{\Delta x_k}\bigg)^2 - \frac{1}{u_k^2}+\frac{\Delta \zeta_k}{\Delta x_k} \bigg]\Delta x_k = \sum_k \frac{(\Delta u_k)^2}{\Delta x_k} - \frac{\Delta x_k}{u_k^2} + \Delta \zeta_k.
\end{equation}
This Lagrangian is not invariant under the product action
\[
X_k = \frac{\alpha x_k + \beta}{\delta x_k + \gamma},\qquad 
U_k = \frac{u_k}{\delta x_k + \gamma},\qquad 
g\cdot \zeta_k = \zeta_k + \frac{\delta u_k^2}{\delta x_k + \gamma},\qquad \alpha \gamma - \beta\delta = 1.
\]
To obtain a symmetry-preserving Lagrangian, we construct a moving frame.  Consider the cross-section
\[
\mathcal{K} = \{x_k = 0, u_k = u_{k+1} = 1\},
\]
which is equivalent to $\mathcal{K} = \big\{x_k=0, u_k=1, u_x^\td=\frac{\Delta u_k}{\Delta x_k}=0\big\}$.  Solving the normalization equations $X_k=0, U_k=U_{k+1}=1$, we obtain the moving frame 
\begin{equation}\label{eq:moving frame 2}
\alpha = \frac{1}{u_k},\qquad 
\beta = -\frac{x_k}{u_k},\qquad
\delta = \frac{\Delta u_k}{\Delta x_k},\qquad
\gamma= \frac{u_k\, \Delta x_k - x_k\, \Delta u_k}{\Delta x_k}.
\end{equation}
Invariantizing \eqref{eq:discrete-modified-L}
\[
\iota(\overline{\mathcal{L}}^\td) = \iota_k(\overline{\mathcal{L}}^\td) = \sum_k \frac{\Delta u_k^2}{\Delta x_k} - \frac{\Delta x_k}{u_k u_{k+1}} + \Delta \zeta_k.
\]
The corresponding Euler--Lagrange equations are
\begin{equation}\label{eq: div invariant scheme}
\begin{aligned}
0 &= \bE_u^\td(\iota(\overline{\mathcal{L}}^\td)) = -2\bigg(\frac{\Delta u_k}{\Delta x_k} - \frac{\Delta u_{k-1}}{\Delta x_{k-1}}\bigg) + \frac{1}{u_k^2}\bigg(\frac{\Delta x_k}{u_{k+1}}+\frac{\Delta x_{k-1}}{u_{k-1}}\bigg),\\
0 &= \bE_x^\td(\iota(\overline{\mathcal{L}}^\td)) = \bigg(\frac{\Delta u_k}{\Delta x_k}\bigg)^2 - \bigg(\frac{\Delta u_{k-1}}{\Delta x_{k-1}}\bigg)^2 + \frac{1}{u_k}\bigg(\frac{1}{u_{k+1}}-\frac{1}{u_{k-1}}\bigg).
\end{aligned}
\end{equation}
Applying Noether's Theorem \ref{thm:discrete-Noether}, with the infinitesimal generators \eqref{eq:v-ex2}, we obtain the conserved quantities
\begin{equation}\label{eq: div invariants}
\begin{aligned}
C_1^\td &= (u_x^\td)^2 + \frac{1}{u_k u_{k+1}},\\
C_2^\td &= \frac{x_k+x_{k+1}}{u_k u_{k+1}} + 2u_x^\td \cdot \frac{u_{k+1} x_k - x_{k+1} u_k}{\Delta x_k},\\
C_3^\td &= \frac{x_k x_{k+1}}{u_k u_{k+1}}+\frac{(u_{k+1} x_k - x_{k+1} u_k)^2}{(\Delta x_k)^2}.
\end{aligned}
\end{equation}
These conserved quantities are independent and satisfy
\begin{equation}\label{eq:discrete C-relation}
\frac{(C_2^\td)^2}{4} - C_1^\td C_3^\td+1 = \frac{1}{4}\bigg(\frac{\Delta x_k}{u_k u_{k+1}}\bigg)^2.
\end{equation}
We note that, in the continuous limit, the equality \eqref{eq:discrete C-relation} converges to \eqref{eq:C-relation}.
\end{example}

\section{Numerical Simulations}\label{sec:num-sim}

In this section we conduct numerical tests for the
invariant variational schemes
\eqref{eq:discrete-elastica-EL} and \eqref{eq: div invariant
  scheme}. We also consider a version of
  \eqref{eq:discrete-elastica-EL} where the distance between points is
  constant. Computations were performed using nonlinear solvers from
  {\tt scipy}'s {\tt optimize} module. In particular, in
  Section~\ref{sec:ee} we use {\tt root} with the Jacobian given
  analytically, while in Section~\ref{sec:dl} we use {\tt fsolve}. All
  nonlinear systems are solved up to an absolute tolerance of
  $10^{-13}$.   For a numerical approximation $u_k$, simulating an exact
  solution $u=u(x_k)$, we will examine the error in the $l_{\infty}$
  norm using the formula
  \begin{equation} \label{eq: infty norm}
    \enorm{u_k - u} := \max{j}{\abs{ u_j - u(x_j)}}.
  \end{equation}
In addition, when benchmarking our simulations we use the following definition.
\begin{definition} \label{def:eoc}
  Given two sequences $a(i), b(i)$, the \emph{experimental order of
  convergence} (EOC) is described by
\[
    \EOC{a,b;i} =
    \frac{\operatorname{log}\left(\frac{a(i+1)}{a(i)}\right)}%
    {\operatorname{log}\left(\frac{b(i+1)}{b(i)}\right)}.
\]
In the sequel $a(i)$ represents a sequence of $l_\infty$ errors given by \eqref{eq: infty norm}, while $b(i)$ represents
  either a step size type parameter or the reciprocal of the number of
  steps taken. 
\end{definition}
  
\subsection{Euler Elastica} \label{sec:ee}

From the perspective of symmetry, invariants, and moving frames, the free Euler elastica equation~\eqref{eq: free elastica} was previously considered in \cite{MRHP-2019}.  Using an approach inspired by the group foliation method, \cite{TV-2018}, the $\text{SE}(2)$ invariance of the Euler--Lagrange equations implies that these equations can be re-expressed in terms of discrete curvature, the arc-length function $\ell_{k} = \sqrt{\Delta x_k^2+\Delta u_k^2}$ and their shifts. 
Solving the Euler--Lagrange equations for these two invariants, the solution $z_k=(x_k,u_k)$ to the original problem is found via a ``reconstruction" process requiring the solution of a system of finite difference equations for the (left) moving frame. In this paper we omit this two step process and solve the Euler--Lagrange equations directly for $z_k=(x_k,u_k)$.  From a numerical perspective, it is not a priori clear if the more involved approach used in \cite{MRHP-2019} gives better results.  On the other hand, the approach introduced in this paper is, we believe, more straightforward to implement.

The parametrized solution to the Euler elastica equation~\eqref{eq: free elastica} is
\begin{equation}\label{eq: ex sol free elastica}
x(s) = \sqrt{\frac{2}{\alpha}} E\bigg(\text{am}\bigg(\sqrt{\frac{\alpha}{2}}s,-1\bigg)-1\bigg)-s,\qquad
u(s) = \sqrt{\frac{2}{\alpha}}\, \text{sn}\bigg(\sqrt{\frac{\alpha}{2}}s,-1\bigg),
\end{equation}
where $\text{sn}(u,k)$ is the Jacobian elliptic sine function, $E(u,k)$ is the incomplete elliptic integral of the second kind, and $\text{am}(t,k)$ is the Jacobian amplitude function.  At the discrete level, the Euler--Lagrange equations \eqref{eq:discrete-elastica-EL} provide a nonlinear system of two equations for the unknown $z_{k+2} = (x_{k+2},u_{k+2})$.  Once the initial conditions $z_0, z_1, z_2, z_3$ are fixed using \eqref{eq: ex sol free elastica}, the numerical solution evolves according to \eqref{eq:discrete-elastica-EL} and there is no way to control the distance between consecutive points, which is generally not numerically desirable.  For small values of $\ell_{k-2}$, $\ldots$, $\ell_{k+1}$, the denominators occurring in the Euler--Lagrange equations are very small, to the point of round-off errors dominating the numerical solution when using standard double precision arithmetic.  Therefore, to implement \eqref{eq:discrete-elastica-EL} we multiplied the Euler--Lagrange equations by $(\ell_{k-1}\ell_k)^{5/2}$ to obtain the scaled equations
\begin{equation}\label{eq: eefree scheme}
\begin{aligned} 
0 &=\widetilde{\bE}_x^\td= -D_{k-2}\Delta u_{k-2} \bigg(\frac{\ell_k}{\ell_{k-2}}\bigg)^{5/2} + D_{k-1}(\Delta u_{k-1}+\Delta u_k) - D_k \Delta u_{k+1}\bigg(\frac{\ell_{k-1}}{\ell_{k+1}}\bigg)^{5/2} \\ 
& -\frac{5D_{k-2}^2\Delta x_{k-1}}{4\ell_{k-1}^2}\bigg(\frac{\ell_k}{\ell_{k-2}}\bigg)^{5/2} - \frac{5D_{k-1}^2 \Delta x_{k-1}}{4 \ell_{k-1}^2} + \frac{5D_{k-1}^2\Delta x_k}{4\ell_k^2} + \frac{5D_k^2\Delta x_k}{4\ell_k^2}\bigg(\frac{\ell_{k-1}}{\ell_{k+1}}\bigg)^{5/2},\\ 
0 &=\widetilde{\bE}_u^\td= D_{k-2}\Delta x_{k-2}\bigg(\frac{\ell_k}{\ell_{k-2}}\bigg)^{5/2} - D_{k-1}(\Delta x_{k-1}+\Delta x_k)+D_k\Delta x_{k+1}\bigg(\frac{\ell_{k-1}}{\ell_{k+1}}\bigg)^{5/2} \\
&- \frac{5D_{k-2}^2\Delta u_{k-1}}{4\ell_{k-1}^2}\bigg(\frac{\ell_k}{\ell_{k-2}}\bigg)^{5/2}-\frac{5D_{k-1}^2\Delta u_{k-1}}{4\ell_{k-1}^2} + \frac{5D_{k-1}^2\Delta u_k}{4\ell_k^2} + \frac{5D_k^2\Delta u_k}{4\ell_k^2}\bigg(\frac{\ell_{k-1}}{\ell_{k+1}}\bigg)^{5/2}.
\end{aligned}
\end{equation}
Supplying the Jacobian entries
\begin{align*}
\pp{\widetilde{\bE}_x^\td}{x_{k+2}} &=
\bigg(\frac{\ell_{k-1}}{\ell_{k+1}}\bigg)^{5/2} \bigg(\Delta u_k \Delta u_{k+1} 
- \frac{5 D_k \Delta u_k \Delta x_k}{2\ell_k^2}
+ \frac{5D_k\Delta u_{k+1}\Delta x_{k+1}}{2\ell_{k+1}^2} 
- \frac{25 D_k^2 \Delta x_k \Delta x_{k+1}}{8\ell_k^2\ell_{k+1}^2}\bigg),\\ 
\pp{\widetilde{\bE}_x^\td}{u_{k+2}} &= \bigg(\frac{\ell_{k-1}}{\ell_{k+1}}\bigg)^{5/2}\bigg(-D_k
-\Delta u_{k+1}\Delta x_k
+\frac{5D_k\Delta x_k^2}{2\ell_k^2}
+\frac{5D_k\Delta u_{k+1}^2}{2\ell_{k+1}^2}
-\frac{25 D_k^2\Delta u_{k+1}\Delta x_k}{8\ell_k^2\ell_{k+1}^2}\bigg),\\
\pp{\widetilde{\bE}_u^\td}{x_{k+2}} &= \bigg(\frac{\ell_{k-1}}{\ell_{k+1}}\bigg)^{5/2}\bigg(D_k 
-\Delta u_k\Delta x_{k+1}
- \frac{5 D_k \Delta u_k^2}{2\ell_k^2}
-\frac{5D_k\Delta x_{k+1}^2}{2\ell_{k+1}^2}
-\frac{25D_k^2\Delta u_k \Delta x_{k+1}}{8\ell_k^2\ell_{k+1}^2}\bigg),\\
\pp{\widetilde{\bE}_u^\td}{u_{k+2}} &= \bigg(\frac{\ell_{k-1}}{\ell_{k+1}}\bigg)^{5/2}\bigg(
\Delta x_k\Delta x_{k+1}
+\frac{5D_k\Delta u_k \Delta x_k}{2 \ell_k^2} 
- \frac{5 D_k \Delta u_{k+1} \Delta x_{k+1}}{2\ell_{k+1}^2}
-\frac{25 D_k^2\Delta u_k \Delta u_{k+1}}{8 \ell_k^2\ell_{k+1}^2}\bigg),
\end{align*}
to {\tt root} in {\tt scipy.optimize} yields an ill-conditioned problem.  To
improve the conditioning of the Jacobian matrix, we added to it a small constant multiple of the identity matrix. In our simulations
this constant is $10^{-3}$, and we note that the specific choice of this
constant depends heavily on $\ell_k$. To initialize the scheme we fixed
\begin{equation} \label{eq:linit}
  \ell_0
  =
  \ell_1
  =
  \ell_2
  =
  0.01
  ,
\end{equation}
and set $s_0=-2$ in the exact solution \eqref{eq: ex sol free elastica}.  Substituting the exact solution in \eqref{eq:linit} we solved for $s_1<s_2<s_3$ in order to obtain the initial conditions $z_k=(x(s_k),u(s_k))$, $k=0,1,2,3$. Running the simulation for
$500$ steps we obtain Figure \ref{fig:eefree}, which we compare against
an exact solution where we assume that $\ell_k$ remains uniform for
$500$ steps.
  \begin{figure}[h!]
  \centering
    \includegraphics[width=0.55\textwidth]{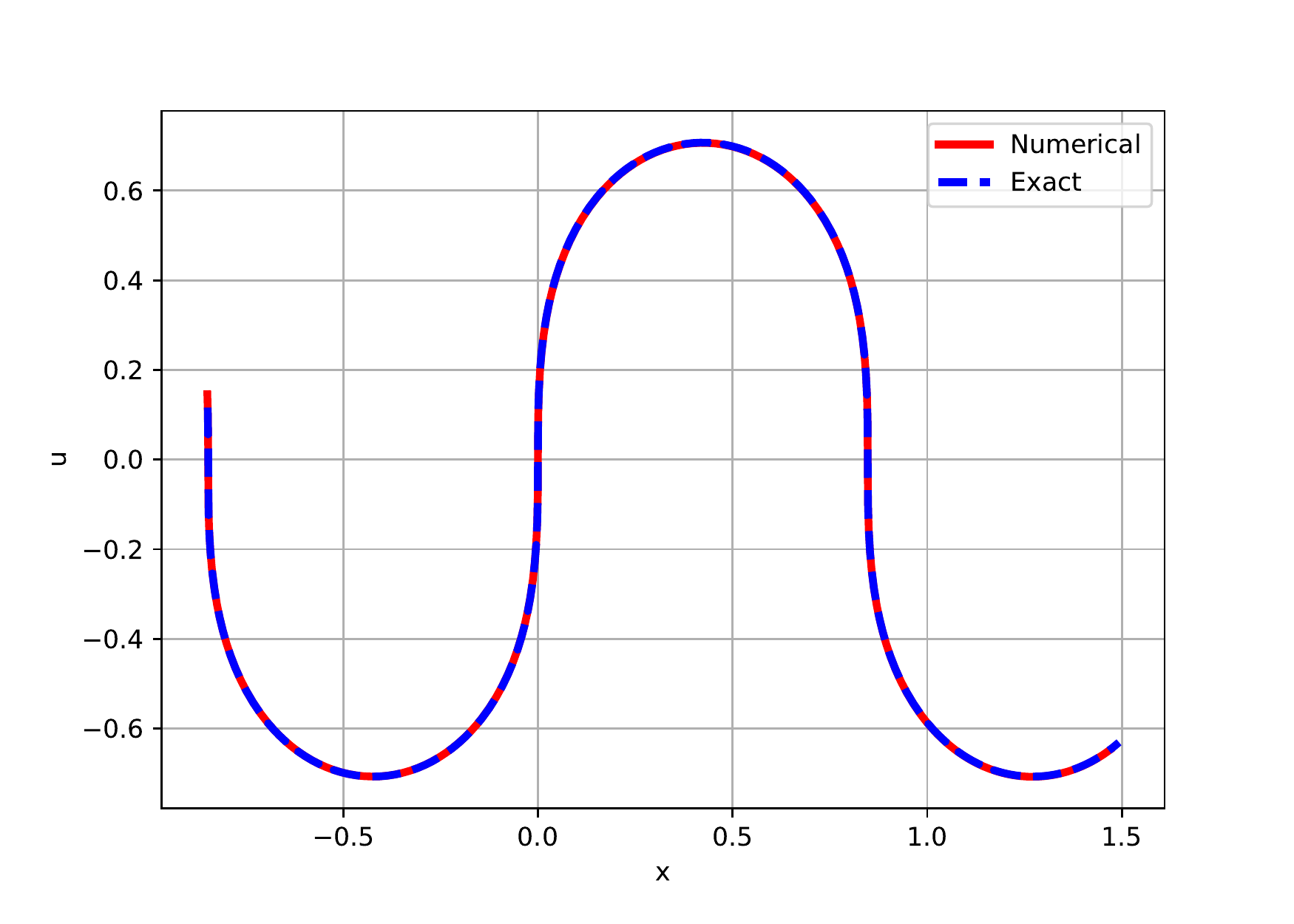}
    \caption{The numerical simulation of solution \eqref{eq: ex
        sol free elastica} using scheme \eqref{eq: eefree
        scheme} superimposed with the exact solution with initial condition $s_0=-2$
      and satisfy
      \eqref{eq:linit}.\label{fig:eefree}}
\end{figure}
We observe that our numerical simulation is qualitatively accurate,
although the numerical solution and the exact solution with uniform
$\ell_k$ evolve at slightly different rates.

To improve on the previous results, and to control the length between neighboring points, we now consider the constrained invariant Lagrangian 
\begin{equation}\label{eq: constrained lagrangian}
L^c_k = \iota(L_k) + \lambda \big(\sqrt{\Delta x_k^2+\Delta u_k^2} - \ell\big),
\end{equation}
where $\ell>0$ is a positive constant, $\lambda$ is a Lagrange multiplier, and $\iota(L_k)$ is given in \eqref{eq:discrete-elastica-invariant-L}.  We note that since $\sqrt{\Delta x_k^2+\Delta u_k^2} - \ell$ is invariant under translations and rotations, the constrained Lagrangrian \eqref{eq: constrained lagrangian} is $\text{SE}(2)$ invariant.  After the multiplication by $\ell^5$, to avoid small denominators, 
the resulting Euler--Lagrange equations are
\begin{equation}\label{eq: constrained scheme}
\begin{aligned}
0=&\ell^5\bE_x^\td(L_k^c) = -D_{k-2} \Delta u_{k-2}+D_{k-1}(\Delta u_{k-1}+\Delta u_k)-D_k\Delta u_{k+1}\\
& + \frac{5}{4\ell^2}[-D_{k-2}^2\Delta x_{k-1} + D_{k-1}^2 (\Delta x_k - \Delta x_{k-1}) +D_k^2 \Delta x_k]-\alpha\mu\,\ell^4(\Delta x_{k-1} - \Delta x_k),\\
0=&\ell^5\bE_u^\td(L_k^c) = D_{k-2}\Delta x_{k-2} - D_{k-1}(\Delta x_{k-1}+\Delta x_k) + D_k\Delta x_{k+1} \\
& + \frac{5}{4\ell^2}[- D_{k-2}^2 \Delta u_{k-1} + D_{k-1}^2 (\Delta u_k -\Delta u_{k-1}) + D_k^2 \Delta u_k]-\alpha\mu\, \ell^4(\Delta u_{k-1}-\Delta u_k),
\end{aligned}
\end{equation}
where we made the substitution $\lambda = -\alpha\mu$.  In the continuous limit, the equations \eqref{eq: constrained scheme} converge to
\begin{equation}\label{eq: limit}
-u_s\bigg(\kappa_{ss} + \frac{\kappa^3}{2}+\alpha\mu\, \kappa\bigg)=0,\qquad
x_s\bigg(\kappa_{ss} + \frac{\kappa^3}{2}+\alpha\mu\, \kappa\bigg)=0,
\end{equation}
respectively.  Therefore, the difference equations \eqref{eq: constrained scheme} provide an approximation of the general Euler elastica equation 
\begin{equation}\label{eq:gen Euler}
\kappa_{ss} + \frac{\kappa^3}{2}+\alpha\mu\, \kappa=0.
\end{equation}
The solution to this ordinary differential equation depends on the value of $\mu$, \cite{DHMV-2008}.  Some of our solutions differ from those appearing in \cite{DHMV-2008}, but have been checked with \texttt{Mathematica} to indeed satisfy the Euler elastica equation:
\begin{description}
\item[$\bm{\mu \in (-1,1)}$:] Let $a = \sqrt{\dfrac{2(1-\mu)}{\alpha}}$ and $c=\sqrt{\dfrac{2(1+\mu)}{\alpha}}$, then
\begin{equation} \label{eqn:case1}
x(s) = c\, E\bigg(\text{am}\bigg(\frac{c\alpha}{2}s,-\frac{a^2}{c^2}\bigg),-\frac{a^2}{c^2}\bigg)-s,\qquad
u(s) = a\, \text{sn}\bigg(\frac{c\alpha}{2}s,-\frac{a^2}{c^2}\bigg),
\end{equation}
where $\text{sn}(u,k)$ is the Jacobian elliptic sine function, $E(u,k)$ is the incomplete elliptic integral of the second kind, and $\text{am}(t,k)$ is the Jacobian amplitude function.

\item[$\bm{\mu = -1}$:] The solution is
\begin{equation} \label{eqn:case2}
x(s) = \frac{2\tanh(\sqrt{\alpha}s)}{\sqrt{\alpha}}-s,\qquad
u(s) = \frac{2\,\text{sech}(\sqrt{\alpha}s)}{\sqrt{\alpha}}.
\end{equation}

\item[$\bm{\mu<-1}$:] Let $a = \sqrt{\dfrac{2(1-\mu)}{\alpha}}$ and $c=\sqrt{-\dfrac{2(1+\mu)}{\alpha}}$, then 
\begin{equation} \label{eqn:case3}
x(s) = c \, E\bigg(\text{am}\bigg(\frac{c\alpha}{2} s,1-\frac{a^2}{c^2}\bigg),1-\frac{a^2}{c^2}\bigg)+\mu s,\qquad
u(s) = c \, \text{dn}\bigg(\frac{c\alpha}{2}s,1-\frac{a^2}{c^2}\bigg),
\end{equation}
where $\text{dn}$ is the delta amplitude function.  
\end{description}

For the numerical implementation of \eqref{eq: constrained scheme}, we note that the two equations are equivalent.  This can be seen by expressing the two equations in the polar coordinates 
\[
\Delta x_k = \ell \cos\theta_k,\qquad \Delta u_k = -\ell\sin\theta_k.
\]
One then finds that
\[
\bE_x^\td(L_k^c)\big[\theta_k+\tfrac{\pi}{2}\big] = \bE_u^\td(L_k^c)\big[\theta_k\big].
\]
To decide which equation from \eqref{eq: constrained scheme} to
choose, we consider their continuous limit \eqref{eq: limit} and note
that when $\Delta u_k \approx u_s$ is close to zero, the first
equation in \eqref{eq: limit} almost vanishes.  Similarly, when
$\Delta x_k \approx x_s$ is close to zero, the second equation in
\eqref{eq: limit} almost vanishes.  Thus, our code for the
implementation of the scheme follows Algorithm \ref{alg:1}.

\begin{algorithm} 
\caption{Constrained Lagrangian Implementation \label{alg:1}}
\begin{algorithmic}
\IF{$|\Delta x_k| < |\Delta u_k|$}
\STATE Solve $\{\ell^5\bE^\td_x(L_k^c)=0,\, \Delta x_{k+1}^2+\Delta u_{k+1}^2-\ell^2 =0\}$ for $(x_{k+2},u_{k+2})$.
\ELSE
\STATE Solve $\{\ell^5\bE^\td_u(L_k^c)=0,\, \Delta x_{k+1}^2+\Delta u_{k+1}^2-\ell^2 =0\}$ for $(x_{k+2},u_{k+2})$.
\ENDIF
\end{algorithmic}
\end{algorithm}

\noindent Notice that once the Euler--Lagrange equation is selected, the second equation used is always $\Delta x_{k+1}^2+\Delta u_{k+1}^2-\ell^2 =0$ to guarantee that the distance between points is constant.  Before sharing our numerical results, we note that the conserved quantities for the Euler--Lagrange equations \eqref{eq: constrained scheme} are
\begin{equation} \label{eq: constrained invariants}
\begin{aligned}
C_1^\td &= -\alpha\mu\,\ell^4\Delta x_{k-1}+ D_{k-1}\Delta u_k - D_{k-2}\Delta u_{k-2} - \frac{5D_{k-1}^2 \Delta x_{k-1}}{4\ell^2}-\frac{5D_{k-2}^2\Delta x_{k-1}}{4\ell^2},\\
C_2^\td &= - \alpha\mu\, \ell^4\Delta u_{k-1} + D_{k-2}\Delta x_{k-2} - D_{k-1}\Delta x_k-\frac{5D_{k-1}^2 \Delta u_{k-1}}{4\ell^2}-\frac{5D_{k-2}^2\Delta u_{k-1}}{4\ell^2},\\
C_3^\td &= x_k C_2^\td - u_k C_1^\td + D_{k-1}(\Delta x_k \Delta x_{k-1}+\Delta u_k \Delta u_{k-1}).
\end{aligned}
\end{equation}

We begin by benchmarking our scheme against the exact solution where
$\mu=-1$ and $\alpha=4$. In this case the exact solution is given by
\eqref{eqn:case2} and it forms a single loop centered at $x=0$ with
$u\to 0$ as $x\to\pm \infty$. We initialize the scheme the same way we
did for \eqref{eq: eefree scheme} using $s_0=-2$ and the fixed value
of $\ell$ to determine the initial data $z_0, z_1, z_2, z_3$ through
the exact solution. While benchmarking we decrease the length $\ell$
and increase the number of steps proportionally to fix the domain of
the simulation. Furthermore, we measure the $l_{\infty}$ error of both
of $x$- and $u$-components of the solution. We observe in Table
\ref{tab:ee} that, experimentally, our scheme is second
order.  This result is interesting as the moving frame \eqref{eq: moving frame 1} used to construct the invariant variational scheme is a first order approximation of its continuous counterpart, \cite{KO-2003}, and the non-invariant Lagrangian \eqref{eq:discrete-elastica-functional} is also a first order approximation of \eqref{eq:elastica-functional}.  This gain in the order of convergence obtained by invariantizing a numerical scheme has also been observed in \cite{K-2007,K-2008}.
\begin{table}[H]
  \centering
    \small{
  \begin{tabular}{|c|c|c||c|c||c|c|}
    \hline
    $i$ & $\ell$ & $\steps$ & $\enorm{x_k - x} =: \err{x}$ & $\EOC{\err{x},\ell;i-1}$ &
                                                                                $\enorm{u_k-u}
                                                                                      =:
                                                                                      \err{u}$
    & $\EOC{\err{u},\ell;i-1}$ \\[0.05cm]
    \hline
    1 & 0.02 & 200 & 2.98e-3 & 9.21 & 1.40e-2 & 6.15 \\
    2 & 0.01 & 400 & 7.14e-4 & 2.03 & 3.41e-3 & 1.98 \\
    3 & 0.005 & 800 & 1.74e-4 & 2.01 & 8.33e-4 & 1.98 \\
    4 & 0.0025 & 1600 & 4.24e-5 & 2.02 & 2.05e-4 & 1.97 \\
    \hline
  \end{tabular}
  }
    \caption{The $l_{\infty}$ error and order of convergence for the invariant variational
      scheme \eqref{eq: constrained scheme} for various values of $\ell$, and where the number of steps is $\frac{2}{\ell}$,
       subject to the exact
      solution \eqref{eqn:case2} with $\mu=-1$ and
      $\alpha=4$.
    }\label{tab:ee}
    \end{table}
\begin{remark}
One might observe that the error for $i=0$ is not displayed in Table \ref{tab:ee}.  If desired, this quantity can be recovered through Definition \ref{def:eoc} after noting that for $i=0$ we fixed $\ell=0.04$.
\end{remark} 
    
For sake of comparison, we also consider the invariant scheme 
\begin{equation}\label{eq: ee naive scheme}
\kappa^\td_{ss} + \frac{(\kappa^\td)^3}{2\ell^4} + \mu\alpha \ell^2\kappa^\td= 0,
\end{equation}
where, up to factors of $\ell$,
\[
\kappa^\td = \Delta^2u_k \Delta x_k - \Delta^2x_k \Delta u_k,\qquad
\kappa_{ss}^\td = \Delta^4u_{k-2} \Delta x_k + \Delta^3u_{k-1} \Delta^2x_k - \Delta^4x_{k-2} \Delta u_k - \Delta^3x_{k-1} \Delta^2u_k,
\]
are approximations of the curvature and its second arc-length derivative with
\begin{align*}
& \Delta z_k = z_{k+1}-z_k,& &
\Delta^2 z_k = z_{k+2}-2z_{k+1} + z_k,\\
&\Delta^3 z_{k-1} = z_{k+2}-3z_{k+1}+3z_k - z_{k-1},& &
\Delta^4 z_{k-2} = z_{k+2} - 4z_{k+1} + 6z_k - 4z_{k-1} + z_{k-2}.
\end{align*}
We note that \eqref{eq: ee naive scheme} is a straightforward discretization of the Euler elastica equation \eqref{eq:gen Euler} where the variational nature of the equation is omitted.  Finally, we supplement \eqref{eq: ee naive scheme} with the equation
\[
\Delta x_{k+1}^2 + \Delta u_{k+1}^2 = \ell^2
\]
to ensure the distance between points is constant.

In Figure \ref{fig:ee:scheme0} we observe that the invariant variational
scheme \eqref{eq: constrained scheme} successfully completes the loop
when $\mu=-1, \alpha=4$, however, the invariant numerical scheme \eqref{eq: ee naive scheme} fails to decay as $x$ decreases to $-\infty$.
\begin{figure}[H]
  \centering
    \includegraphics[width=0.55\textwidth]{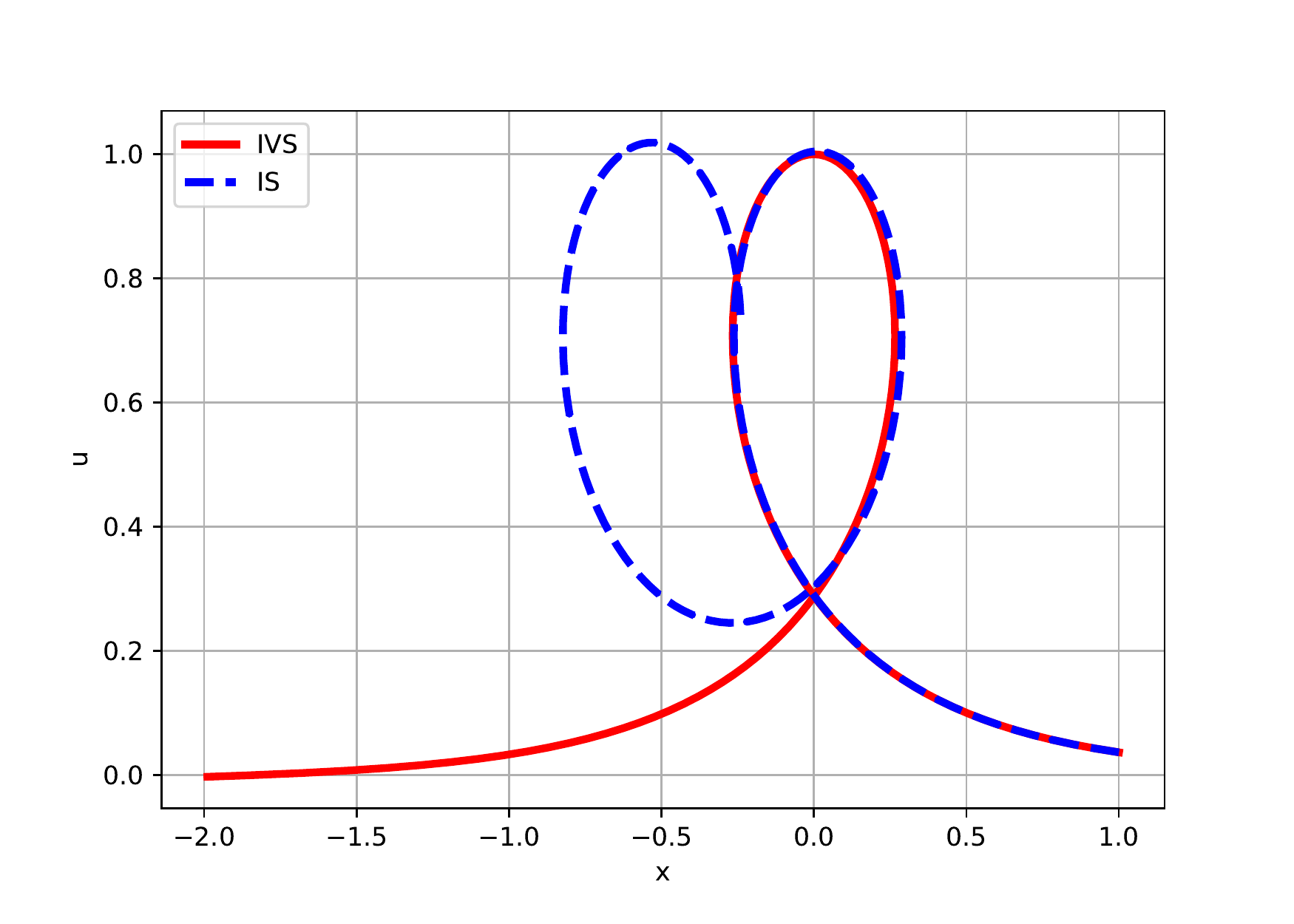}
    \caption{Numerical simulation of solution
      \eqref{eqn:case2}, where $\mu=-1$ and $\alpha=4$, using both
      the invariant variational scheme (IVS) \eqref{eq: constrained scheme} and the
invariant scheme (IS) \eqref{eq: ee naive scheme} with $\ell=0.01$ and the number of steps equal to 500. \label{fig:ee:scheme0}}
  \end{figure}
In Figure \ref{fig:ee:c:scheme0} we plot the deviation of the conserved quantities \eqref{eq: constrained invariants}.
\begin{figure}[H]
  \centering
  \subfigure[][Invariant variational scheme]{
    \includegraphics[width=0.45\textwidth]{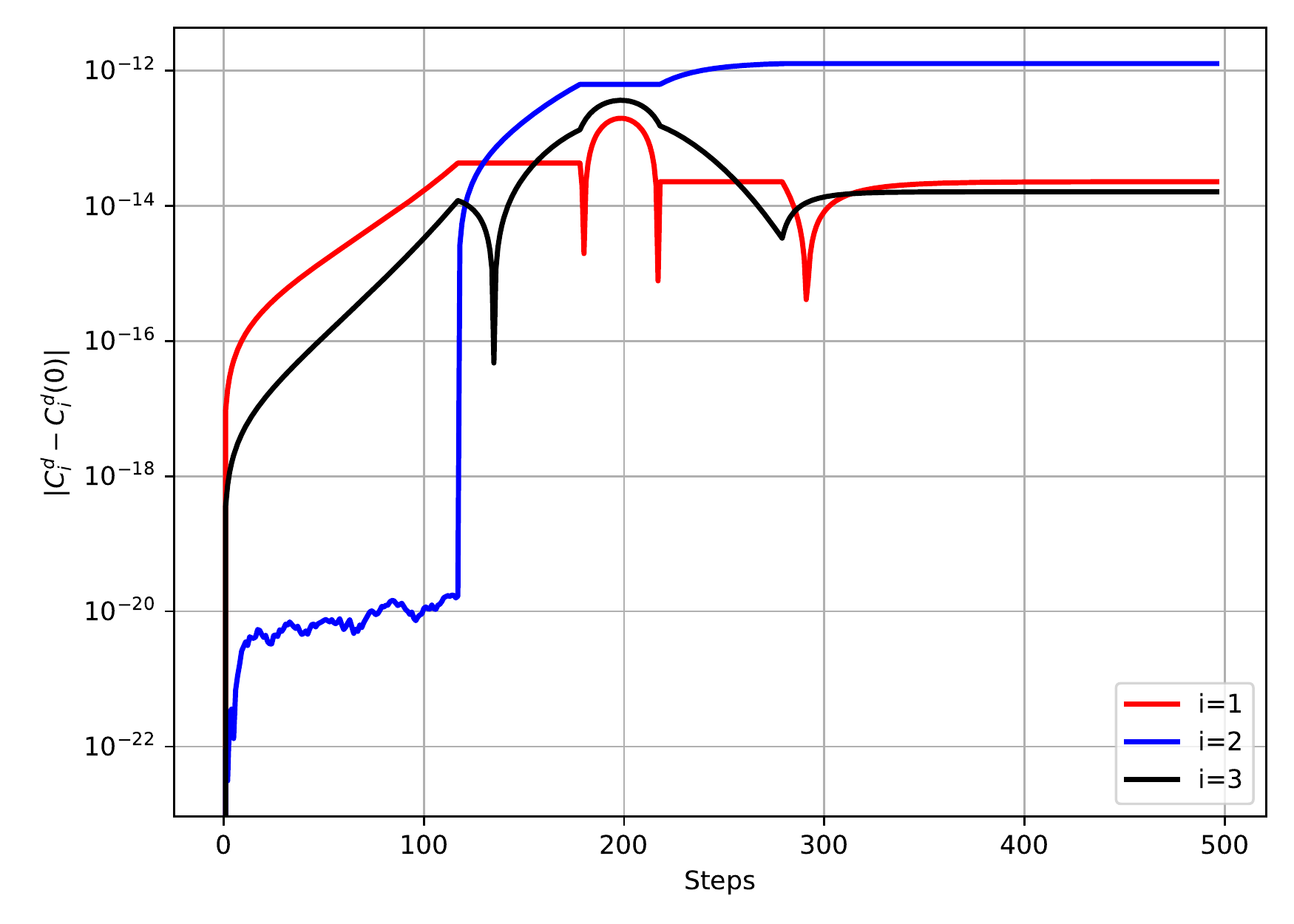}
  }
  \subfigure[][\revise{Invariant scheme}]{
    \includegraphics[width=0.45\textwidth]{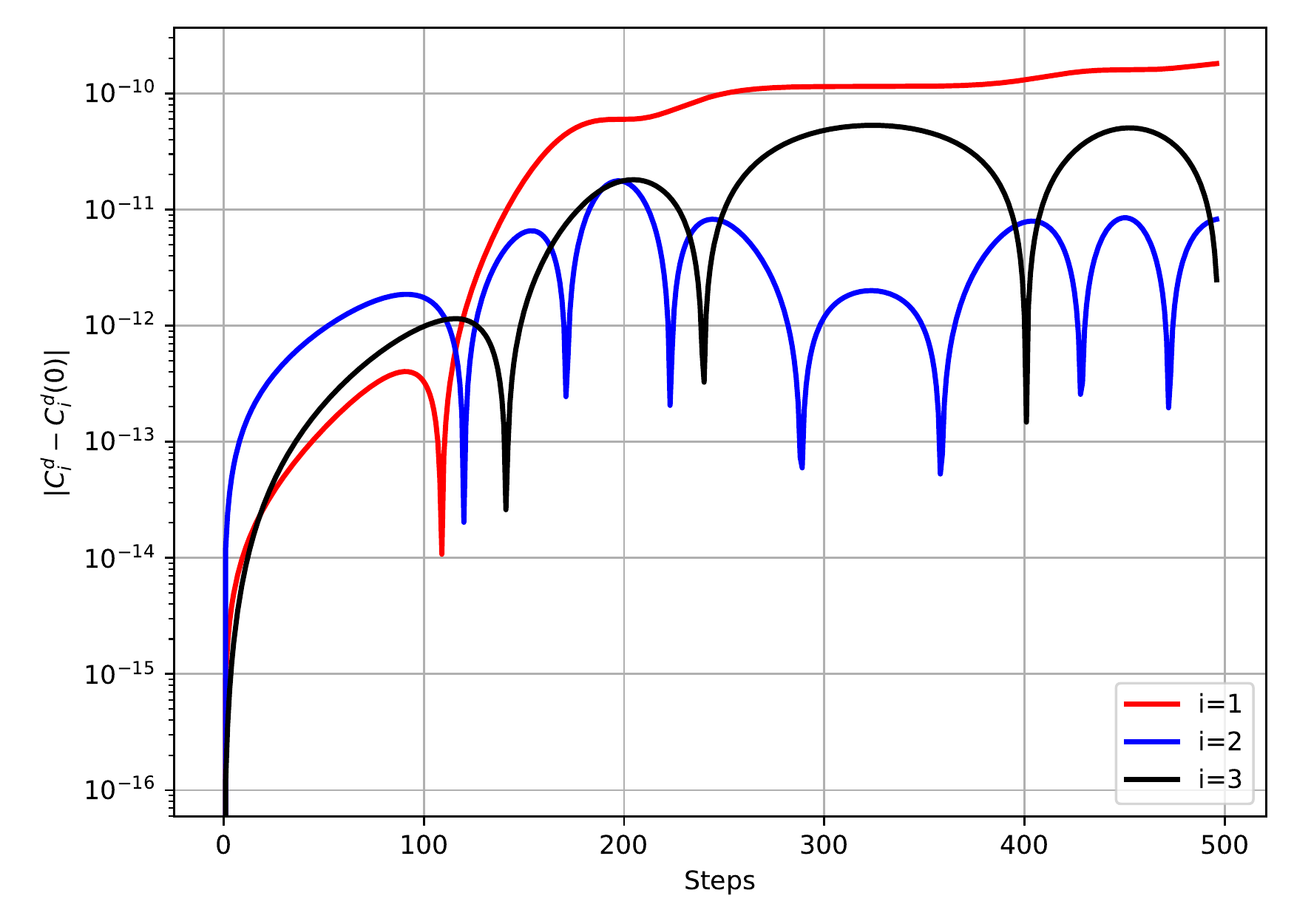}
  }
  \caption{The deviation in the conserved quantities $C_i^\td$, $i=1,2,3$, as
    described by \eqref{eq: constrained invariants}, for the
    invariant variational scheme \eqref{eq: constrained scheme} and
    invariant scheme \eqref{eq: ee naive scheme} simulating \eqref{eqn:case2}
    with $\mu=-1$, $\alpha=4$, and where $\ell=0.01$.
    \label{fig:ee:c:scheme0}}
\end{figure}
We observe that over time the deviation in the conservative scheme \eqref{eq: constrained scheme} propagates to $10^{-12}$ while the invariant scheme \eqref{eq: ee naive scheme} propagates to $10^{-10}$. We note that this deviation remains small due to the order of magnitude of the
conserved quantities themselves, however, we do not look at the relative
deviation here as the errors propagating below solver precision become
significant.

While the schemes \eqref{eq: constrained scheme} and \eqref{eq: ee naive scheme} are both invariant under the special Euclidean group action, the above simulation shows that the scheme which is also variational, and therefore preserves the constants of motion, provides better long term numerical results.  

  For completeness, we also consider a non-invariant
  variational scheme obtained by computing the Euler--Lagrange equations
  of the non-invariant Lagrangian \eqref{eq:discrete-elastica-functional} subject to the 
  constraint $\sqrt{\Delta x_k^2+\Delta u_k^2}=\ell$.  The resulting equations
  are
\begin{equation}\label{eq:noninvariant scheme}
\begin{aligned} 
0 &= \ell^5\bE_u(L_k^c) = \bigg(D_k \Delta x_{k+1} + \frac{5D_k^2\Delta u_k}{2\ell^2}\bigg)\bigg(\frac{\Delta x_k}{\Delta x_{k+1}}\bigg)^{5/2}+ (D_{k-2}\Delta x_{k-2})\bigg(\frac{\Delta x_{k-2}}{\Delta x_{k-1}}\bigg)^{5/2} \\
&-\bigg(D_{k-1}(\Delta x_k + \Delta x_{k-1}) + \frac{5 D_{k-1}^2 \Delta u_{k-1}}{2\ell^2}\bigg)\bigg(\frac{\Delta x_{k-1}}{\Delta x_k}\bigg)^{5/2} - \ell^4\alpha \mu(\Delta u_{k-1}-\Delta u_k),\\
 0 &= \ell^5 \bE_x(L_k^c) = \bigg(-D_k\Delta u_{k+1} + \frac{5D_k^2\Delta x_k}{2\ell^2}-\frac{5D_k^2}{4\Delta x_k}\bigg)\bigg(\frac{\Delta x_k}{\Delta x_{k+1}}\bigg)^{5/2}\\
                           &+\bigg(D_{k-1}(\Delta u_k + \Delta u_{k-1})-\frac{5D_{k-1}^2\Delta x_{k-1}}{2\ell^2}+\frac{5D_{k-1}^2}{4 \Delta x_{k-1}}+\frac{5D_{k-1}^2}{4\Delta x_k}\bigg)\bigg(\frac{\Delta x_{k-1}}{\Delta x_k}\bigg)^{5/2} \\
                           &-\bigg(D_{k-2}\Delta u_{k-2} + \frac{5D_{k-2}^2}{4\Delta x_{k-1}}\bigg)\bigg(\frac{\Delta x_{k-2}}{\Delta x_{k-1}}\bigg)^{5/2}-\ell^4\alpha\mu(\Delta x_{k-1}-\Delta x_k),
\end{aligned}
\end{equation}  
to which we add the constraint equation  $\Delta x_{k+1}^2 + \Delta u_{k+1}^2 - \ell^2=0$.
As for the invariant variational scheme \eqref{eq: constrained
    scheme},  we employ Algorithm \ref{alg:1} to choose the
  optimal combination of the equations to solve. Replicating the experiments showcased in
  Figure \ref{fig:ee:scheme0} and Figure \ref{fig:ee:c:scheme0}, we
  obtain Figure \ref{fig:ee:ni:scheme0}.
  \begin{figure}[H]
    \centering
    \subfigure[][]{
      \includegraphics[height=5.5cm]{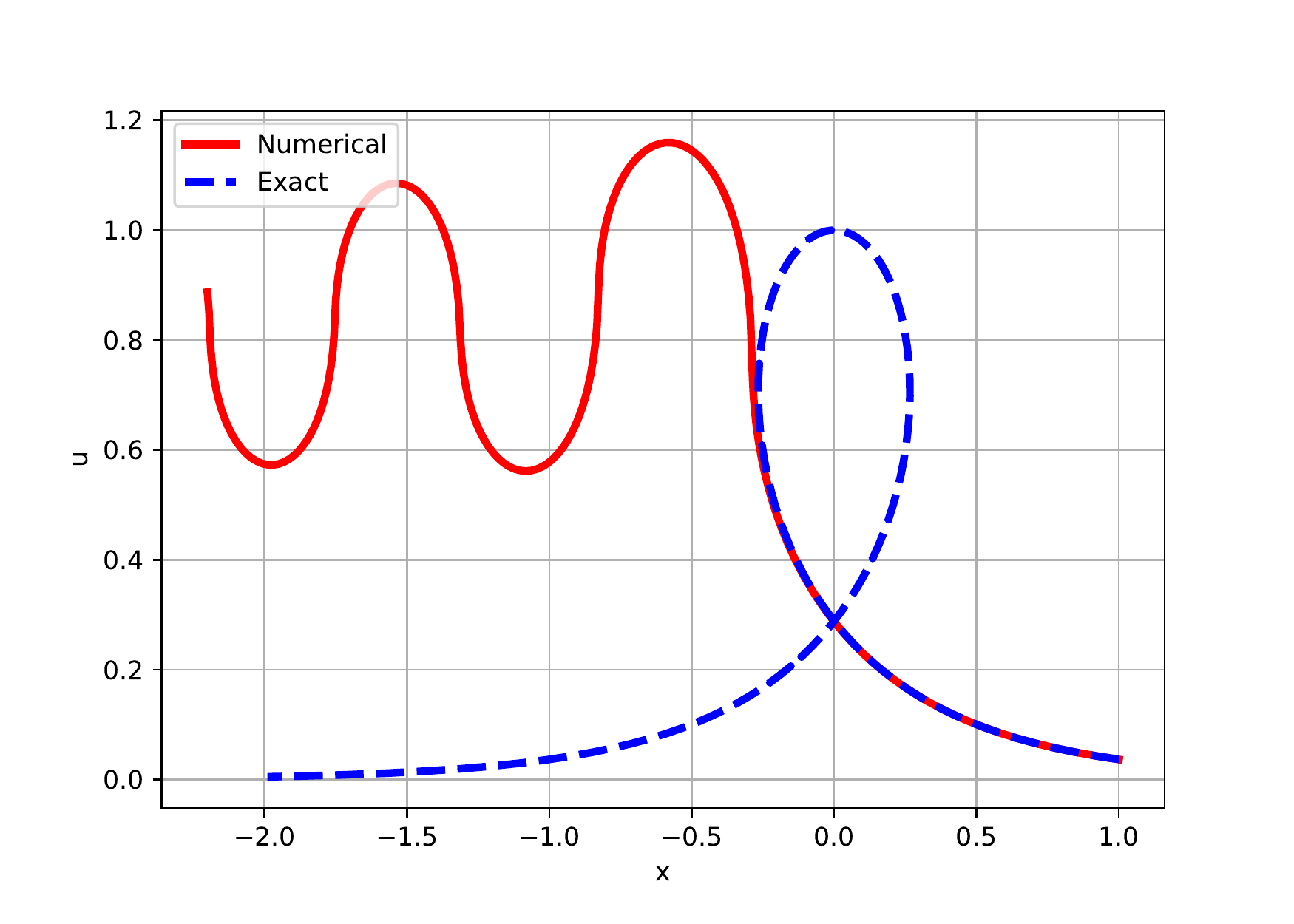}
    }
    \subfigure[][]{
      \includegraphics[height=5cm]{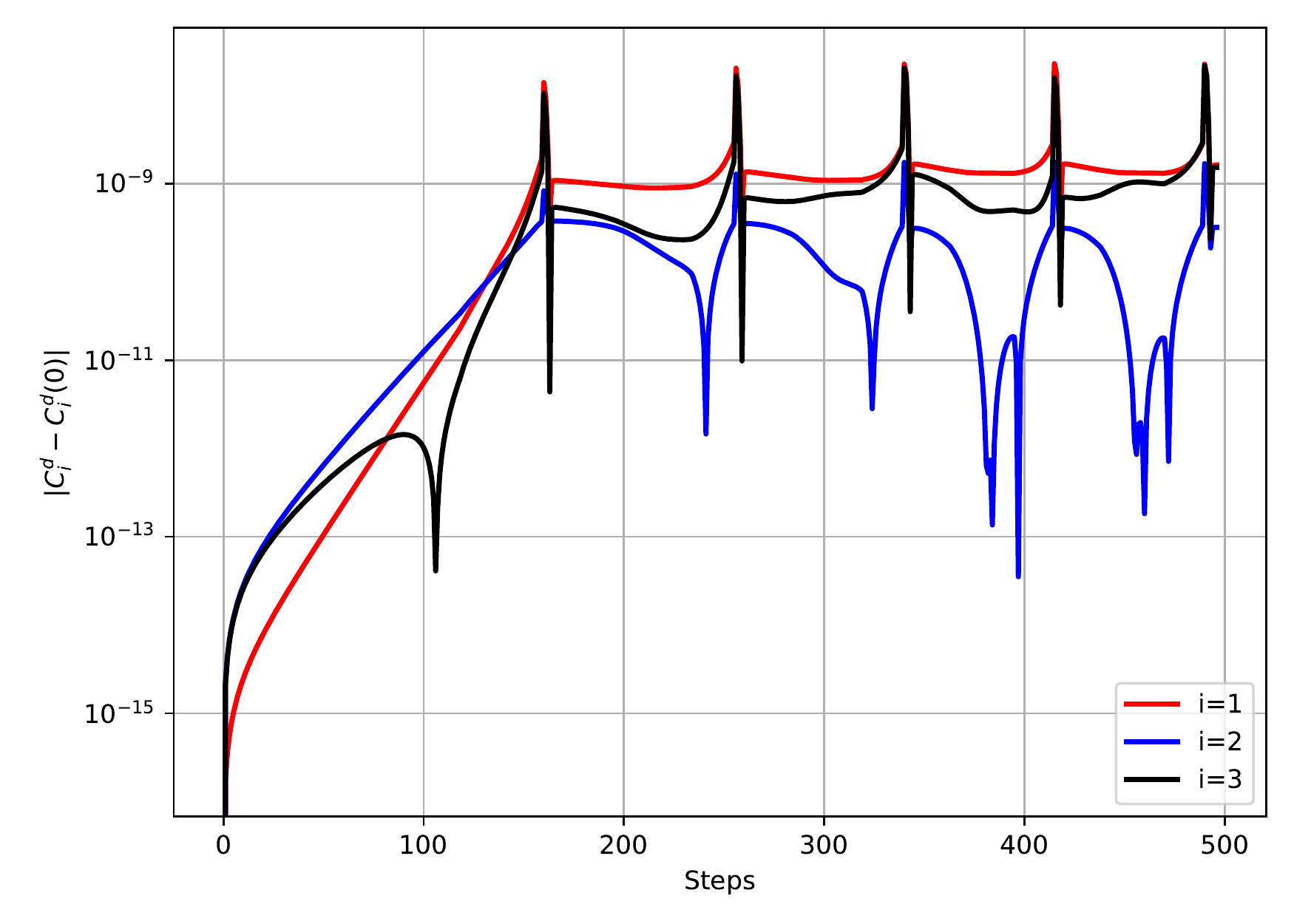}
    }
    \caption{(a) Numerical simulation of solution
      \eqref{eqn:case2}, where $\mu=-1$ and $\alpha=4$, using the
      non-invariant scheme \eqref{eq:noninvariant scheme} compared
      against the exact solution.
      (b) The deviation in the conserved quantities $C_i^\td$, $i=1,2,3$, as
      described by \eqref{eq: constrained invariants}, for the
      non-invariant scheme \eqref{eq:noninvariant scheme}.
      For both simulations we fix $\ell=0.01$ and the number of $\steps=500$.
      \label{fig:ee:ni:scheme0}}
  \end{figure}
  We observe that the solution to the non-invariant scheme diverges when the 
  tangent line to the curve becomes vertical. This was to be expected since in 
  \eqref{eq:noninvariant scheme} the Euler--Lagrange equations are divided by
  $\Delta x_k$ (and $\Delta x_{k+1}$, $\Delta x_{k-1}$).  Furthermore the deviation
  in the conserved quantities is orders of magnitude greater than the invariant
  variational scheme \eqref{eq: constrained scheme}.  

There are a multitude of interesting dynamics exhibited by the Euler
elastica equation which may be simulated by our model. Fixing
$\ell=0.01$ and iterating for 1000 steps, we obtain Figure \ref{fig:dynamics} for different values of $\mu$.
\begin{figure}[h!]
  \centering
  \subfigure[][$\mu=-1.2$]{
    \includegraphics[width=0.31\textwidth]{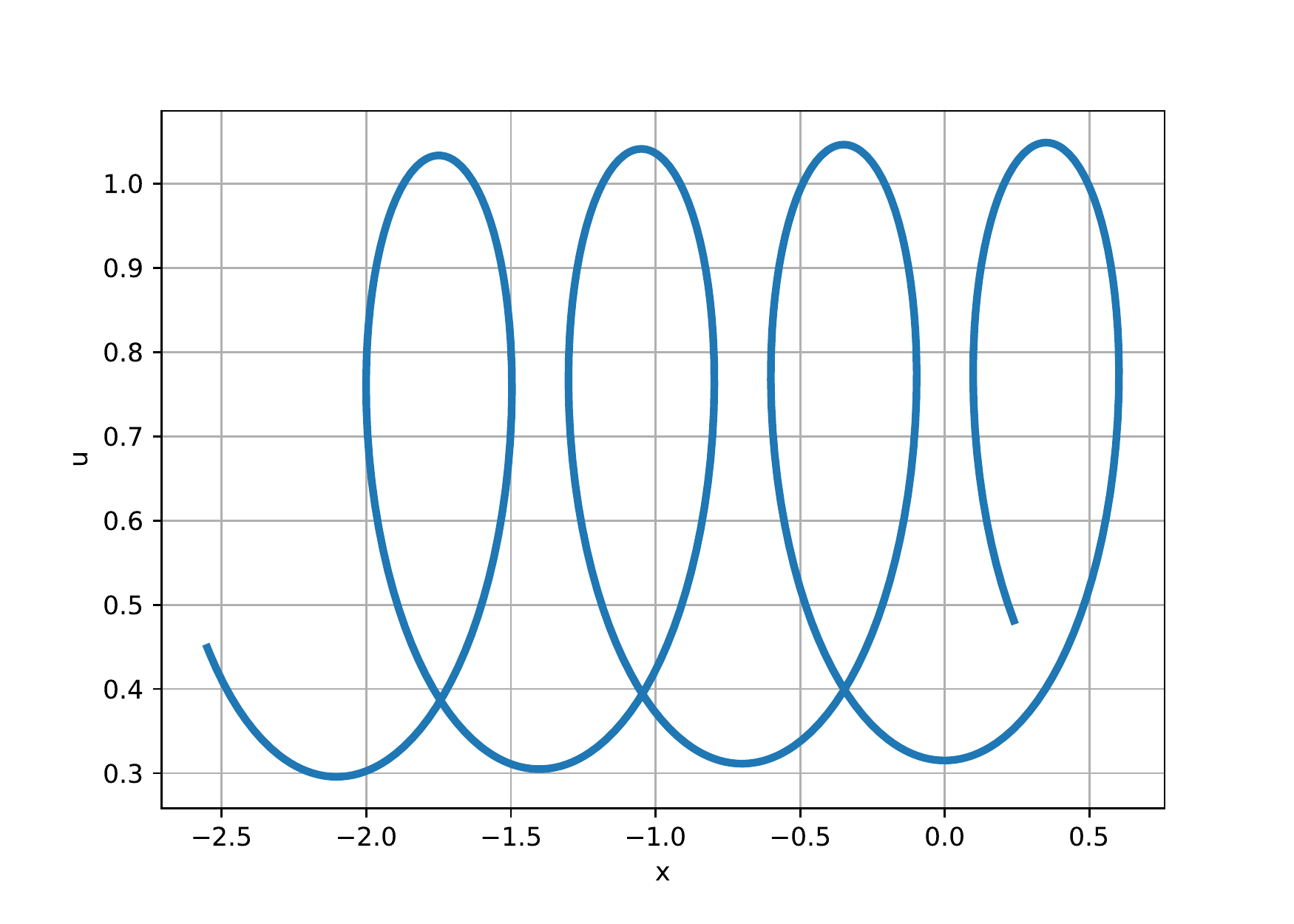}
  }
  \subfigure[][$\mu=0.5$]{
    \includegraphics[width=0.31\textwidth]{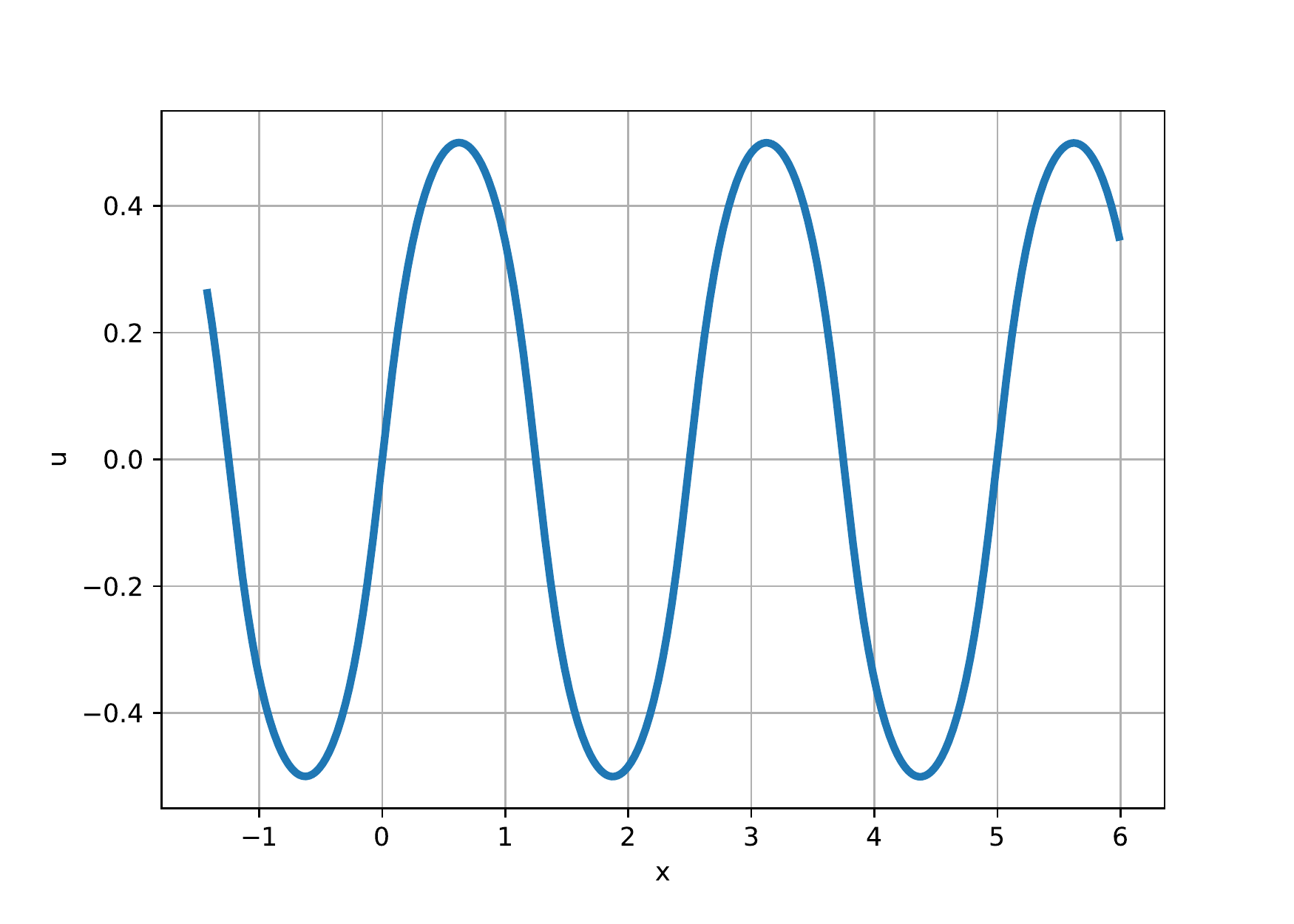}
  }
  \subfigure[][$\mu=0$]{
    \includegraphics[width=0.31\textwidth]{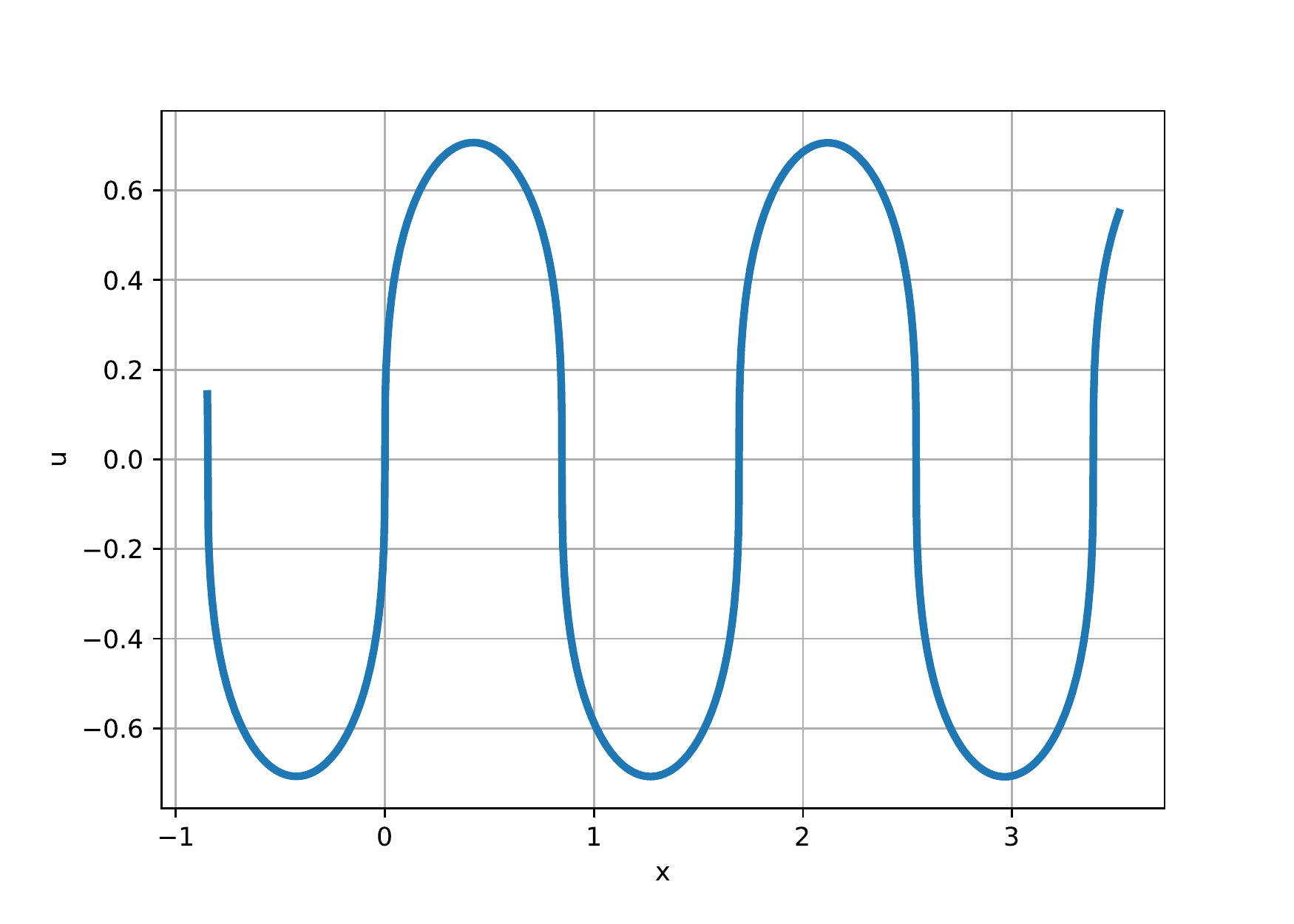}
  }
    \subfigure[][$\mu=-0.4$]{
    \includegraphics[width=0.31\textwidth]{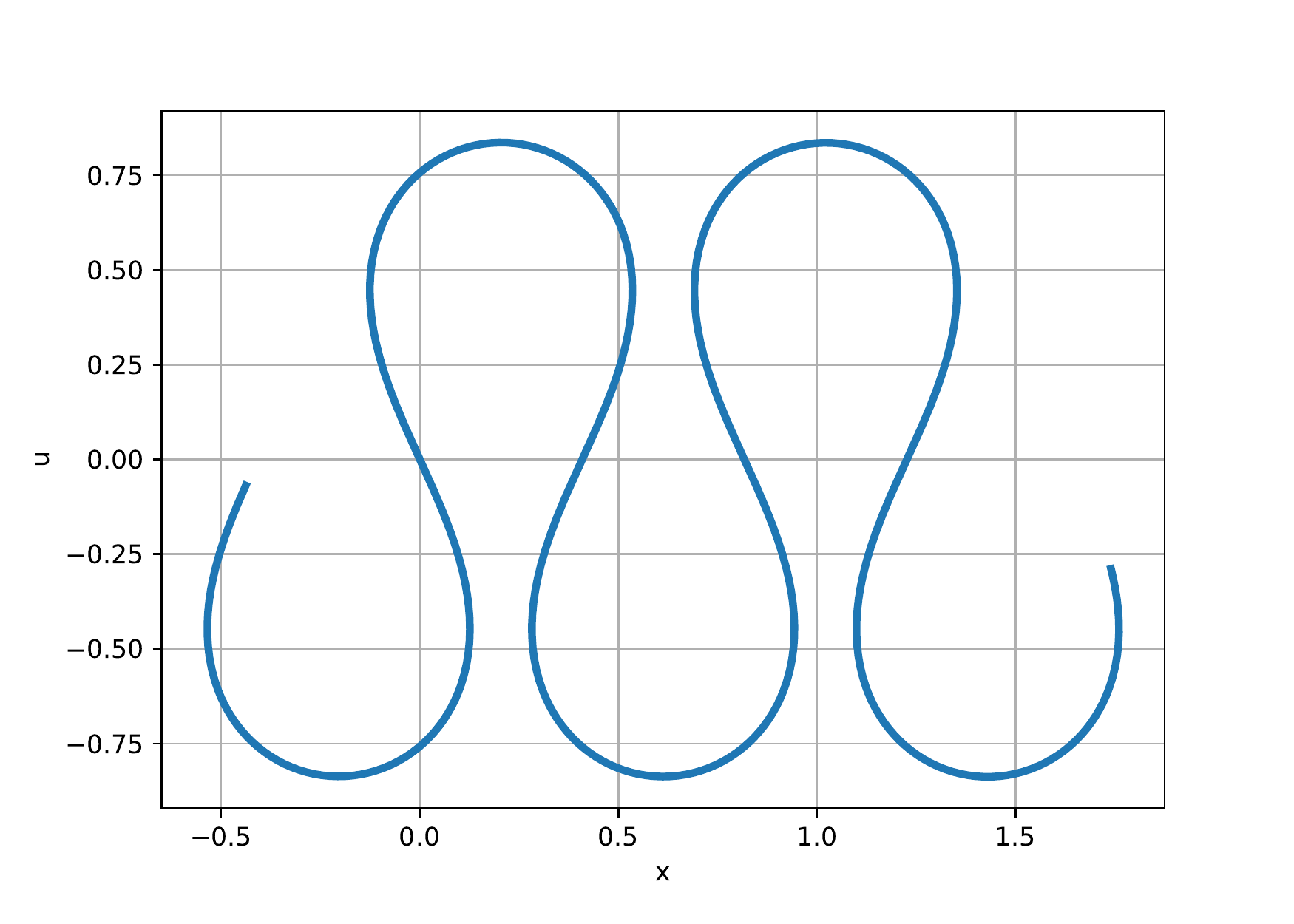}
  }
  \subfigure[][$\mu=-0.65223$]{
    \includegraphics[width=0.31\textwidth]{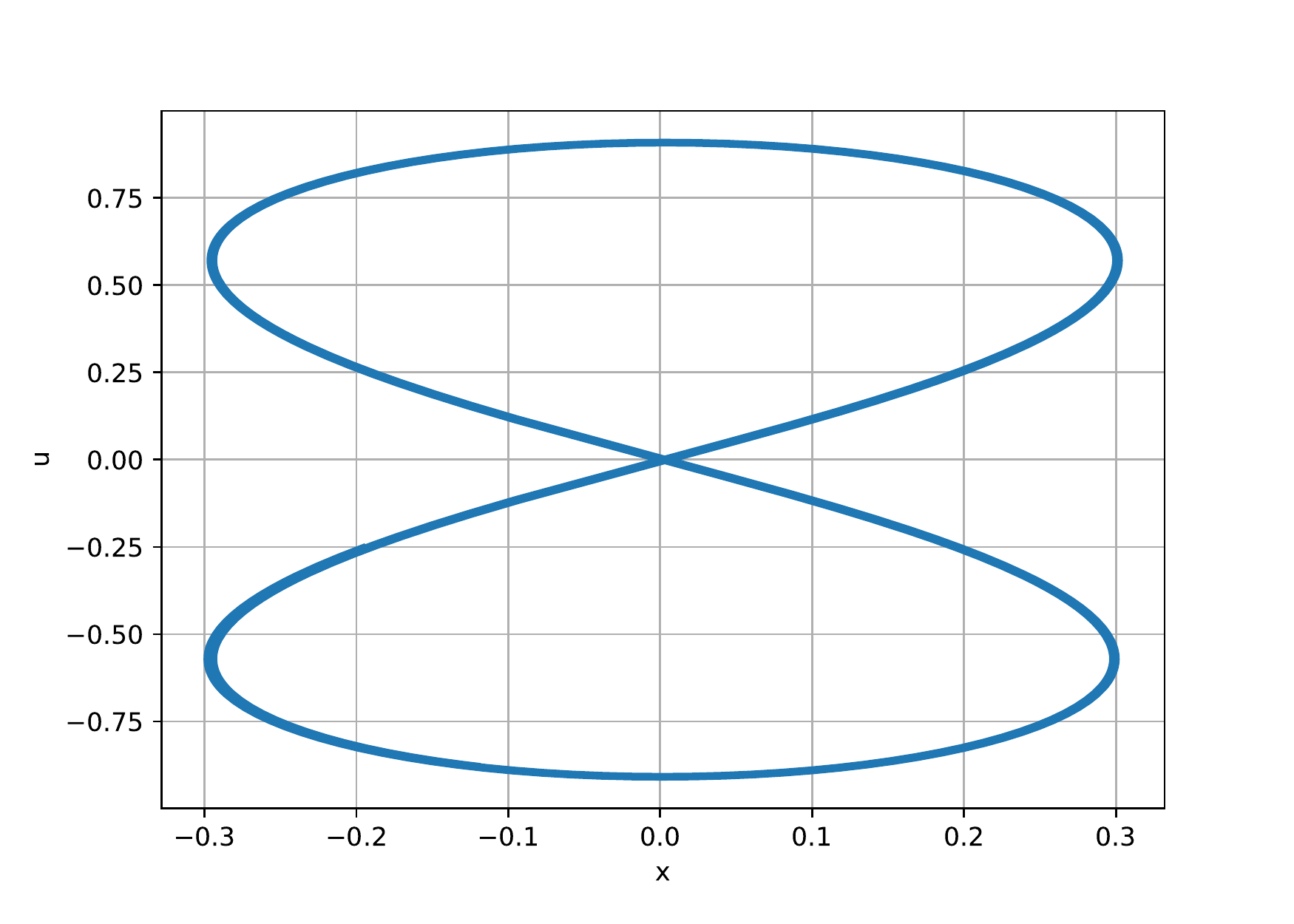}
  }
  \subfigure[][$\mu=-0.9$]{
    \includegraphics[width=0.31\textwidth]{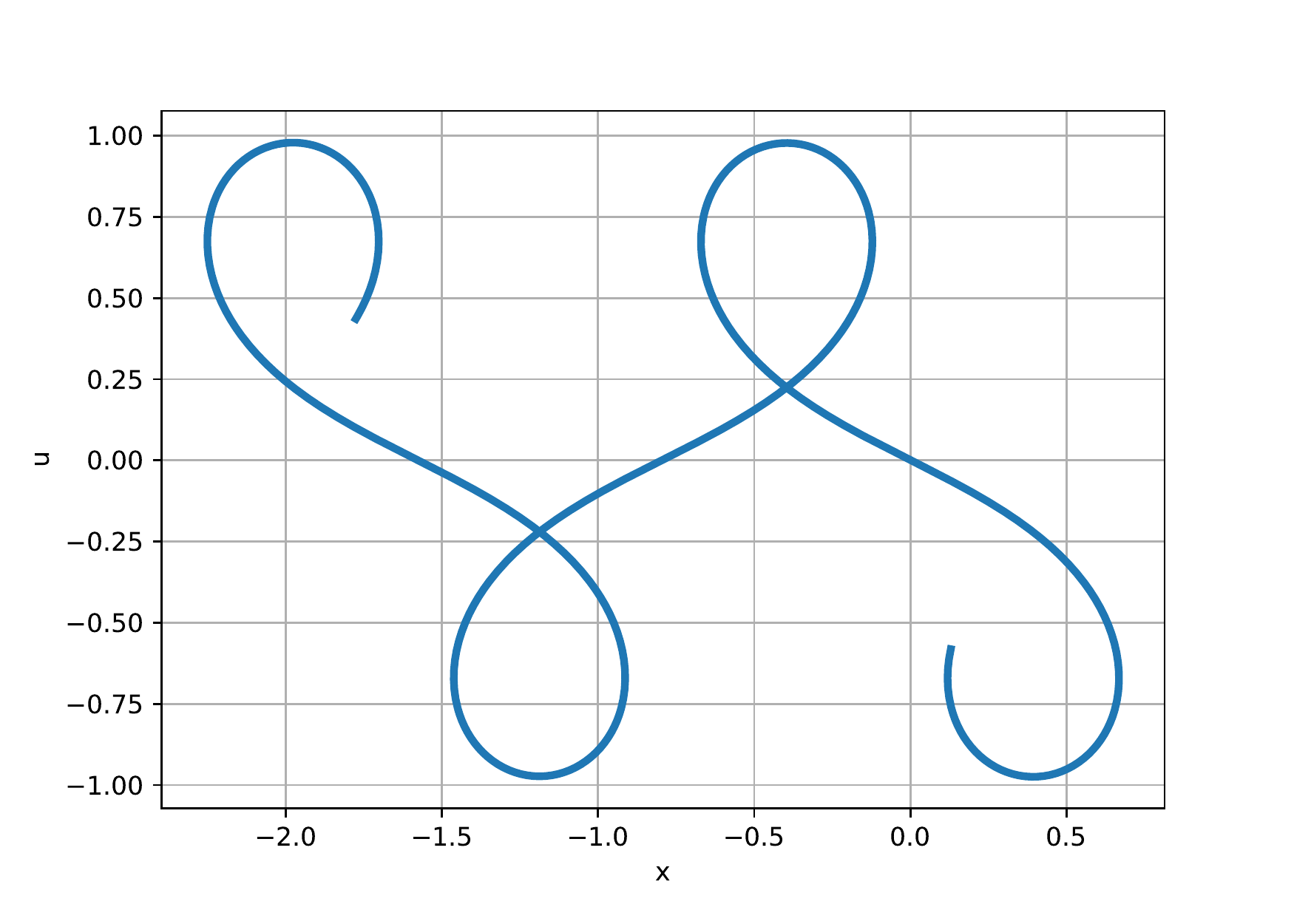}
  }
  \caption{The invariant variational scheme \eqref{eq:
      constrained scheme} with $\alpha=4$ for various $\mu$ values simulating the solutions described in \eqref{eqn:case1}
    and \eqref{eqn:case3}. We initialize all simulations at $s_0=-2$ and
    iterate 1{,}000 steps with $\ell=0.01$. We note that all dynamics
    presented in this figure are accurate, when compared to the exact solutions.
    \label{fig:dynamics}}
\end{figure}

\FloatBarrier
\subsection{Divergence Invariant Lagrangian} \label{sec:dl}

We now shift our focus to the divergence invariant scheme \eqref{eq:
  div invariant scheme}. For comparison, we compare our
invariant approximation against a standard approximation of \eqref{eq:divergent EL} given by
\begin{equation} \label{eq: div standard scheme}
\frac{u_{k+1}-2u_k+u_{k-1}}{h^2} = \frac{1}{u_{k}^3},
\end{equation}
where $\Delta x_k = h$ is constant. We note that the general solution to \eqref{eq:divergent EL} is
\begin{equation} \label{eq: div invariant exact}
  u(x) = \sqrt{ \frac{\left(A x + B \right)^2 + 1}{A}},
\end{equation}
where $A$ and $B$ are constants. For our simulations, we consider the
case where $A=1$ and $B=0$. As the numerical solution $(x_k,u_k)$
evolves according to \eqref{eq: div invariant scheme}, the component
$u_k$ will provide an approximation of exact solution $u(x_k)$.
Therefore, when benchmarking our approximation we only consider the
error in the $u$
component. 
We initialize our simulation by setting $x_0=-1$,
$x_1=-1 + \frac{2}{\steps-1}$ and $u_0=u(x_0)$, $u_1 = u(x_1)$.
Benchmarking our numerical approximation \eqref{eq: div invariant
  scheme} in Table \ref{tab:div} we obtain a quadratic experimental
order of convergence.  Additionally, by design, the
  non-invariant scheme \eqref{eq: div standard scheme} also converges to
  second order.  We note that the quadratic convergence of the invariant variational scheme is better than expected, as the modified Lagrangian
  \eqref{eq:discrete-modified-L} is first order accurate and the
  discrete moving frame \eqref{eq:moving frame 2} is also a first
  order approximation of its continuous counterpart.  This indicates
  that, in this example, an order of accuracy has been gained through
  the invariantization procedure.

\begin{table}[h!]
  \centering
  \scriptsize{
  \begin{tabular}{|c|c||c|c|}
    \hline
    i & $\steps$ &  $\enorm{u_k-u} =: \err{u}$ & $\EOC{\err{u},\steps^{-1};i-1}$ \\[0.05cm]
    \hline
    1 & 200 & 2.32e-6 & 2.04 \\
    2 & 400 & 5.73e-7 & 2.02 \\
    3 & 800 & 1.42e-7 & 2.01 \\
    4 & 1600 & 3.55e-8 & 2.00 \\
    \hline
  \end{tabular}
  \hspace{0.25cm}
  \begin{tabular}{|c|c||c|c|}
    \hline
    $i$ & $\steps$ &  $\err{u}$  & $\EOC{\err{u},\steps^{-1}; i-1}$ \\[0.05cm]
    \hline
    1 & 200 & 8.50e-6 & 2.01 \\
    2 & 400 & 2.11e-6 & 2.00 \\
    3 & 800 & 5.27e-7 & 2.00 \\
    4 & 1600 & 1.32e-7 & 2.00 \\
    \hline
  \end{tabular}
  }
  \caption{The $l_{\infty}$ error and the order of convergence in the $u$
    component for the invariant variational scheme \eqref{eq: div
      invariant scheme} (left) and the standard numerical approximation
    \eqref{eq: div standard scheme} (right) approximating solution \eqref{eq:
      div invariant exact} with $A=1$ and $B=0$.
    \label{tab:div}
  }
\end{table}

Fixing the number of steps to 100, we simulate both the invariant and
standard schemes and compute the conserved quantities in Figure
\ref{fig:devinv}. We note that deviation in the conserved
quantities changes slowly for the invariant variational scheme
remaining on the order of the solver precision ($10^{-13}$),
while for the standard scheme all quantities deviate
  significantly above machine precision with $C_2^\td$ reaching
  $\mathcal{O}\left(10^{-4}\right)$ by the end of the simulation.
\begin{figure}[h!]
  \centering
  \subfigure[][Invariant variational scheme]{
    \includegraphics[width=0.45\textwidth]{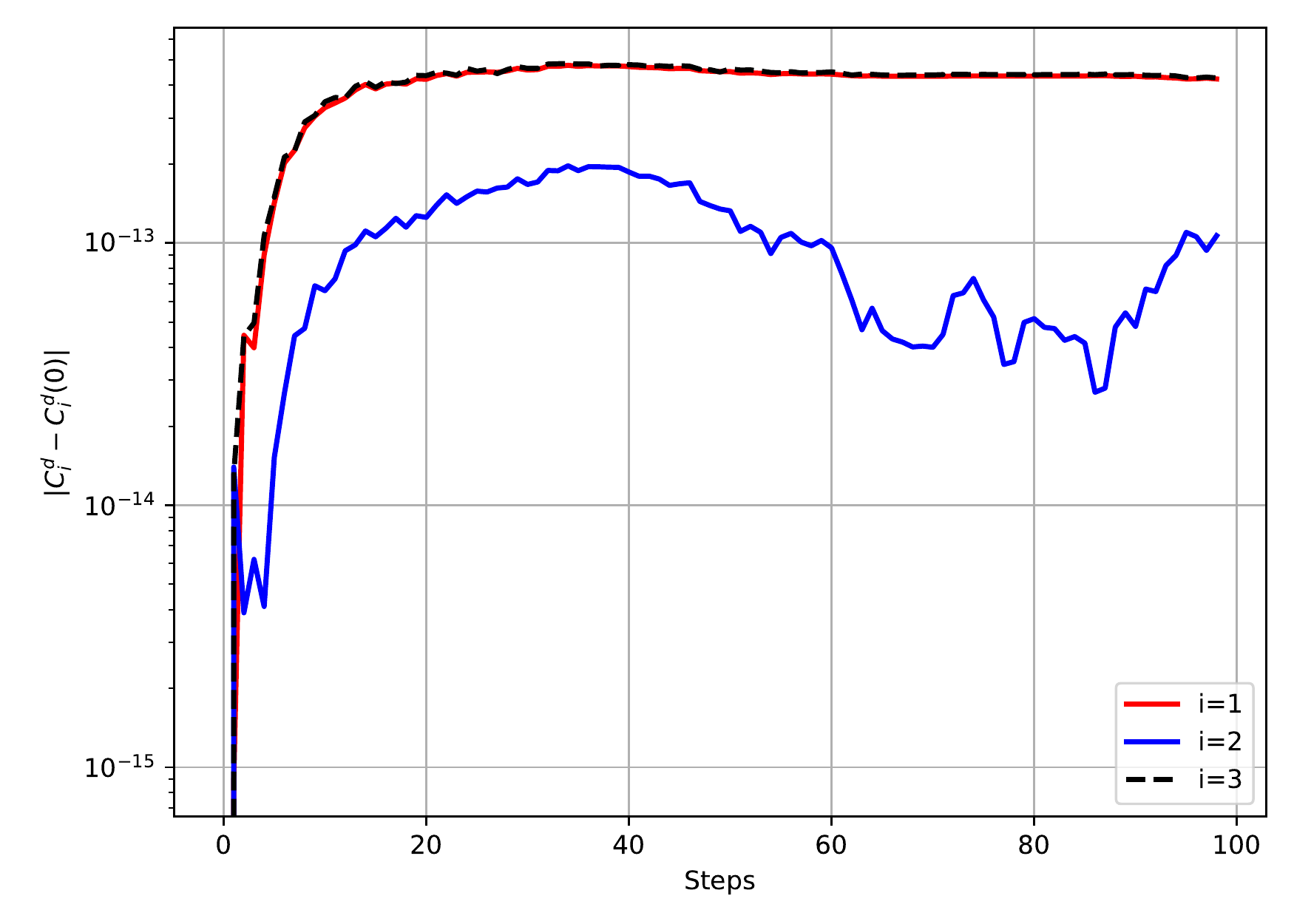}
  }
  \subfigure[][Standard scheme]{
    \includegraphics[width=0.45\textwidth]{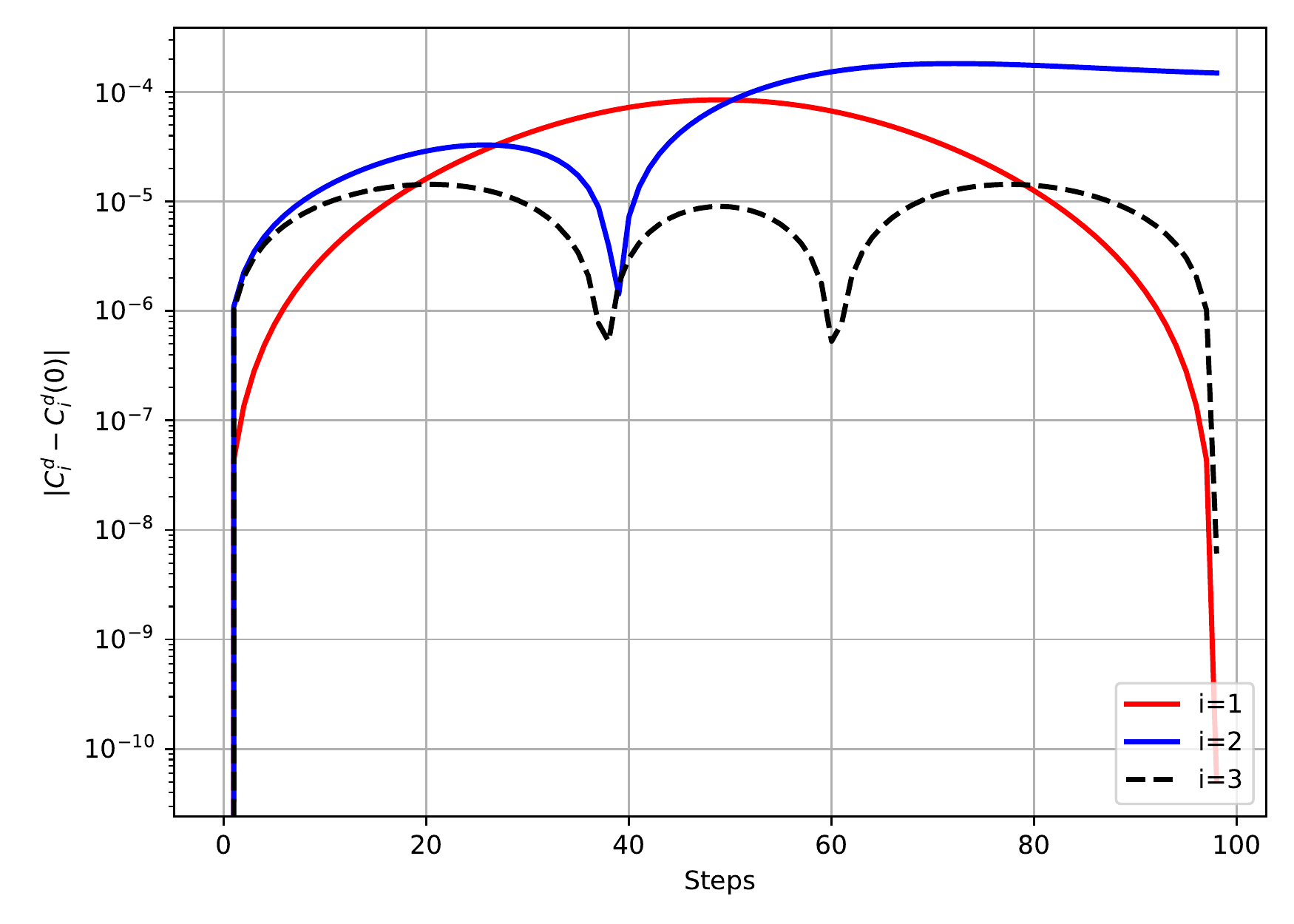}
  }
  \caption{The deviation in the conserved quantities $C_i^\td$, $i=1,2,3$, as described by
    \eqref{eq: div invariants}, for the invariant variational scheme \eqref{eq:
      div invariant scheme} and the non-invariant scheme \eqref{eq: div
      standard scheme}. The simulations are initialized with the exact
    solution at $x_0=-1$ and $x_1=-1-2/99$, and 100 steps are
    implemented. We observe that the deviations for the
    invariant variational scheme remain on the order of solver precision while
    for the non-invariant scheme the deviations quickly propagate.
    \label{fig:devinv}
    }
\end{figure}

\FloatBarrier

\section{Conclusion}

Given a system of ordinary differential equations, one can use the Helmholtz conditions to determine whether or not these coincide with the Euler--Lagrange equations of some Lagrangian, \cite{AT-1992}.  For ordinary differential equations that originate from a variational problem, we introduced a procedure for discretizing the equations so as to preserve both its variational (and divergence) symmetries and its conserved quantities.   This is done in a three step process where we first discretize the continuous Lagrangian to obtain a discrete variational problem.  During this discretization procedure, Lie point symmetries are usually lost.  To recover the lost symmetries we  implement the moving frame method and invariantize the discrete Lagrangian.  The numerical scheme is then obtained by computing the Euler--Lagrange equations of the invariantized Lagrangian.

The invariant variational approach outlined in the previous paragraph offers several advantages over other related geometric integrators.   First, compared to invariant integrators, \cite{B-2013,BJV-2020,BN-2013,BN-2014,BV-2019}, that only focus on preserving the symmetries of the Euler--Lagrange equations, without consideration to its variational origin, the invariant variational schemes constructed in this paper have the additional benefit of preserving the conserved quantities of the problem. By preserving first integrals, the schemes should be more stable and produce better long term numerical results, which is one of the main appealing properties of geometric numerical integrators.  Next, compared to the conservative method introduced in \cite{WBN-2016,WBN-2017}, our construction is simpler to implement and avoids the use of divided difference calculus, which can become challenging at times.  Similarly, the discrete gradient method introduced in \cite{QM-2008}, which requires recasting the system in a skew-gradient form, is nontrivial to implement, in particular for large dynamical systems with many first integrals.  On the other hand, in our approach one can naively discretize a Lagrangian and recover a suitable symmetry-preserving Lagrangian via the algorithmic process of invariantization.

Finally, we note that the methodology developed in this paper can also be applied to partial differential equations.    As for ordinary differential equations, the discrete Euler--Lagrange equations will simultaneously approximate the differential equation and provide equations for the mesh.  Though, as with any symmetry-preserving integrators, the mesh equations might lead to mesh entangle and poor numerical results.  To alleviate these issues one could possibly use invariant $r$-adaptive meshes, \cite{BP-2012} or evolution--projections techniques, \cite{BN-2013}, adapted to the variational framework.  Doing so would require more attention, and we therefore reserve this problem for future considerations elsewhere.

\section*{Acknowledgements}
This research was undertaken, in part, thanks to funding from the Canada Research Chairs program, the InnovateNL LeverageR{\&}D program and the NSERC Discovery program.


\begin{thebibliography}{99}
\bibitem{AT-1992}
Anderson, I., and Thompson, G., {\it The Inverse Problem of the Calculus of Variations for Ordinary Differential Equations}, Memoirs of the American Mathematical Society {\bf 473}, AMS, Providence, 1992.

\bibitem{B-1921}
Bessel-Hagen, E., \"Uber die Erhaltungss\"atze der Elektrodynamik, {\it Math.\ Ann.} {\bf 84} (1921), 258--276.

\bibitem{BC-2016}
Blanes, S., and Casas, F., {\it A Concise Introduction to Geometric Numerical Integration},  Monographs and Research Notes in Mathematics, Vol.\ 23, CRC Press, 2016.

\bibitem{B-2013}
Bihlo, A., Invariant meshless discretization schemes, {\it J.\ Phys.\ A} {\bf 46} (2013), 062001.

\bibitem{BJV-2020}
Bihlo, A., Jackaman, J., and Valiquette, F., On the development of symmetry-preserving finite element schemes for ordinary differential equations, {\it J.\ Comp.\ Dyn.} {\bf 7} (2020), 339--368.

\bibitem{BN-2013}
Bihlo, A., and Nave, J.-C., Invariant discretization schemes using evolution-projection techniques, {\it SIGMA} {\bf 9} (2013), 052.

\bibitem{BN-2014}
Bihlo, A., and Nave, J.-C., Convecting reference frames and invariant numerical models, {\it J.\ Comput.\ Phys.} {\bf 271} (2014), 656-663.

\bibitem{BP-2012}
Bihlo, A., and Popovych, R.O., Invariant discretization schemes for the shallow-water equations, {\it SIAM J.\ Sci.\ Comput.} {\bf 34} (2012), B810--B839.

\bibitem{BV-2017}
Bihlo, A., and Valiquette, F., Symmetry-preserving numerical schemes, in {\it Symmetries and Integrability of Difference Equations}, CRM Ser.\ Math.\ Phys., Springer (2017), 261--324.

\bibitem{BV-2019}
Bihlo, A., and Valiquette, F., Symmetry-preserving finite element schemes. An introductory investigation, {\it SIAM J.\ Sci.\ Comput.} {\bf 41} (2019), A3300--A3325.

\bibitem{B-2002}
Boutin, M., On orbit dimensions under a simultaneous Lie group action on $n$ copies of a manifold, {\it J.\ Lie Theory} {\bf 12} (2002), 191--203.

\bibitem{DHMV-2008}
Djondojorov, P.A., Hadzhilazova, M.T., Mladenov, I.M., and Vassilev, V.M., {\it Explicit parametrization of Euler's Elastica}, Ninth International Conference on Geometry, Integrability, and Quantization, June 80013, 2007, Varna, Bulgaria, Iva\"ilo M.\ Mladenov, Editor {\bf SOFTEX}, Sofia 2008, 175--186.

\bibitem{E-1744}
Euler, L., Additamentum `De Curvis Elasticis', in {\it Methodus Inveniendi Lineas Curvas Maximi Minimive Proprietate Gaudentes}, Lausanne, 1744.

\bibitem{FO-1999}
Fels, M., and Olver, P.J., Moving coframes: II. Regularization and theoretical foundations, {\it Acta Appl.\ Math.} {\bf 55} (1999), 127--208.

\bibitem{HLW-2006} 
Hairer, E., Lubich, C., and Wanner, G., {\it Geometric Numerical Integration: Structure-Preserving Algorithms for Ordinary Differential Equations}, vol.\ 31, 2nd edition, Springer, 2006.

\bibitem{K-2007}
Kim, P., Invariantization of numerical schemes using moving frames, {\it BIT Num.\ Math.} {\bf 47} (2007), 525--546.

\bibitem{K-2008}
Kim, P., Invariantization of the Crank-Nicolson method for Burgers' equation, {\bf 237} (2008), 243--254.

\bibitem{KO-2004}
Kim, P., and Olver, P.J., Geometric integration via multi-space, {\it Regul.\ Chaotic Dyn.} {\bf 9} (2004), 213--226.

\bibitem{KO-2003}
Kogan, I.A., and Olver, P.J., Invariant Euler--Lagrange equations and the invariant variational bicomplex, {\it Acta Appl.\ Math.} {\bf 76} (2003), 137--193.

\bibitem{LR-2004}
Leimkuhler, B., and Reich, S., {\it Simulating Hamiltonian Dynamics}, vol.\ 14 of Cambridge Monographs on Applied and Computational Mathematics, Cambridge University Press, 2004.

\bibitem{L-1927}
Love, A.E.H, {\it The Mathematical Theory of Elasticity}, Cambridge University Press, London, 1927.

\bibitem{MMW-2013}
Mansfield, E.L., Mari-Beffa, G., and Wang, J.P., Discrete moving frames and discrete integrable systems, {\it Found.\ Comp.\ Math.} {\bf 13} (2013), 545--582.

\bibitem{MRHP-2019}
Mansfield, E.L., Rojo-Echbur\'ua, A., Hydon, P.E., and Peng, L., Moving frames and Noether's finite difference conservation laws I, {\it Trans.\ Math.\ Appl.} {\bf 3} (2019), 1--47.

\bibitem{MM-2018}
Mari-Beffa, G., and Mansfield, E.L., Discrete moving frames on lattice varieties and lattice-based multispaces, {\it Found.\ Comp.\ Math.} {\bf 18} (2018), 181--247.

\bibitem{MW-2001}
Marsden, J.E, and West, M., Discrete mechanics and variational integrators, {\it Acta Numer.} {\bf 10} (2001), 357--514.

\bibitem{O-2001-1}
Olver. P.J., Geometric foundations of numerical algorithms and symmetry, {\it Appl.\ Alg.\ Engin.\ Comp.\ Commun.} {\bf 11} (2001), 417--436.

\bibitem{O-2001}
Olver. P.J., Joint invariant signatures, {\it Found.\ Comp.\ Math.} {\bf 1} (2001), 3-68.

\bibitem{O-1993}
Olver, P.J., {\it Applications of Lie Groups to Differential Equations}, Second Edition, Graduate Texts in Mathematics, Vol.\ 107, Springer, New York, 1993.

\bibitem{QM-2008}
Quispel, G., and McLaren, D., A new class of energy-preserving numerical integration methods, {\it J.\ Phys.\ A: Math.\ Theor.} {\bf 41} (2008), 045206.

\bibitem{SC-1994}
Sanz-Serna, J., and Calvo, M., {\it Numerical Hamiltonian Problems}, Applied Mathematics and Mathematical Computation, Vol.\ 7, Chapman \& Hall, 1994.

\bibitem{TV-2018}
Thompson, R., and Valiquette, F., Group foliation of finite difference equations, {\it Commun.\ Nonlinear Sci.\ Numer.\ Simul.} {\bf 59} (2018), 235--254.

\bibitem{WBN-2016}
Wan, A.T.S, Bihlo, A., and Nave, J.-C., The multiplier method to construct conservative finite difference schemes for ordinary and partial differential equations, {\it SIAM J.\ Numer.\ Anal.} {\bf 54} (2016), 86--119.

\bibitem{WBN-2017}
Wan, A.T.S., Bihlo, A., and Nave, J.-C., Conservative methods for dynamical systems, {\it SIAM J.\ Numer.\ Anal.} {\bf 55} (2017), 2255--2285.

\bibitem{ZM-1988}
Zhong, G., and Marsden, J., Lie--Poisson, Hamilton--Jacobi theory and Lie--Poisson integrators, {\it Phys.\ Lett.\ A} {\bf 133} (1988), 134--139.

\end{thebibliography}
\end{document}